\documentclass[11pt]{amsart}
\usepackage[latin1]{inputenc}
\usepackage{amsmath,latexsym,amssymb,graphicx,dsfont,pstricks}
\usepackage{natbib}
\usepackage{color}

\renewcommand{\r}{\mathbf{r}}

\newcommand{\C}{\mathbb{C}}

\newtheorem{Th}{Theorem}[section]
\newtheorem{Lemma}[Th]{Lemma}
\newtheorem{Cor}[Th]{Corollary}
\newtheorem{Prop}[Th]{Proposition}

\newcommand{\Prob}{\mathbb{P}}

\renewcommand{\O}{\mathcal{O}}
\renewcommand{\o}{\mathbf{o}}
\newcommand{\T}{\mathcal{T}}
\newcommand{\N}{\mathbb{N}}
\newcommand{\Z}{\mathbb{Z}}

\numberwithin{equation}{section}

\begin{document}

\title[Phase Transitions for Random Walk Asymptotics on Free Products]{Phase Transitions for Random Walk Asymptotics on Free Products of Groups}
%\author{Lorenz A. Gilch\footnotemark[1]}
%\author[Lorenz A. Gilch]{}

\address{Institut f\"ur mathematische Strukturtheorie (Math. C), University of Technology Graz, Steyrergasse
  30, A-8010 Graz, Austria}

\email{candellero@tugraz.at, gilch@TUGraz.at}

\urladdr{http://www.math.tugraz.at/$\sim$candellero/, http://www.math.tugraz.at/$\sim$gilch/}
\date{\today}
\subjclass[2000]{Primary 60J10, 60F99; Secondary 58K55} 
\keywords{Random Walks, Free Products, Return Probabilities,
  Asymptotic Behaviour, Lattices}

\maketitle

\centerline{\scshape Elisabetta Candellero and Lorenz A. Gilch}
\medskip
{\footnotesize
 \centerline{Graz University of Technology, Graz, Austria}}

\begin{abstract}
Suppose we are given finitely generated groups $\Gamma_1,\dots,\Gamma_m$
equipped with irreducible random walks. Thereby we assume that the expansions of the corresponding Green functions
at their radii of convergence contain only logarithmic or algebraic terms as
singular terms up to sufficiently large order (except for some degenerate cases). We consider transient random walks on the free product \mbox{$\Gamma_1 \ast
\ldots \ast\Gamma_m$} and give a complete classification of the possible
asymptotic behaviour of the corresponding $n$-step return probabilities. 
They either inherit a law of the form  $\varrho^{n\delta} n^{-\lambda_i} \log^{\kappa_i}n$ from one of
the free factors $\Gamma_i$ or obey a $\varrho^{n\delta} n^{-3/2}$-law, where $\varrho<1$ is
the corresponding spectral radius and $\delta$ is the period of the random walk. In addition, we determine the full range of
the asymptotic behaviour in the case of nearest neighbour random walks on free
products of the form $\Z^{d_1}\ast \ldots \ast \Z^{d_m}$. Moreover, we
characterize the possible phase transitions of the non-exponential types
$n^{-\lambda_i}\log^{\kappa_i}n$ in the case $\Gamma_1\ast\Gamma_2$.
\end{abstract}

\section{Introduction}

In this article we investigate transient random walks on free products $\Gamma_1\ast\ldots \ast \Gamma_m$, where $m\geq 2$ and
$\Gamma_1,\dots,\Gamma_m$ are finitely generated groups. These random
walks arise from convex combinations of probability measures
on the factors $\Gamma_1,\dots,\Gamma_m$. Our aim is to compute the asymptotic
behaviour of the $n$-step return probabilities on the free product. In a general setting, one has
a typical asymptotic behaviour of the form
$\mu^{(n)}(x) \sim C_x\,\varrho^{n\delta} n^{-\lambda}$, where $\mu^{(n)}(x)$ is the
probability of being at $x$ at time $n$, $\varrho$ is the spectral radius, $\delta$ the period of the random walk, and $C_x$ some constant depending on $x$. If $e$ is the group identity and starting point of the random walk, then $\mu^{(n)}(e)$ is called the $n$-step return probability. 
Gerl \cite{gerl2} conjectured that
the $n$-step return probabilities of two symmetric measures on a group
satisfying such a limit law have the
same $n^{-\lambda}$, that is, $\lambda$ is a group invariant.
Cartwright \cite{cartwright} came to the astonishing result that for random walks on
$\Z^d\ast \Z^d$ with $d\geq 5$ there are at least two possible types of
asymptotic behaviour, namely $n^{-3/2}$ and $n^{-d/2}$. In relation with his joint work with Chatterji and Pittet \cite{chatterji-pittet-saloffcoste}, L.\,Saloff-Coste asked whether the range of different asymptotic behaviours can still be wider than in the case considered by Cartwright. In this article we will pick up this question by investigating more general laws of the form $C\,\varrho^{n\delta}\,
 n^{-\lambda}\,\log^{\kappa} n$. In this case, we speak of the factor
 $n^{-\lambda}\,\log^{\kappa} n$ as the non-exponential type of the return probabilities.
\par
The starting point for the present investigation was Woess \cite[Section
17.B]{woess}, where the result of Cartwright \cite{cartwright} is explained that simple random walk on $\Z^d\ast \Z^d$ for
$d\geq 5$ satisfies a $n^{-d/2}$-law. In this article we will prove the following more general theorem:
\begin{Th}\label{thm:1.1}
Let $m\in\N$ with $m\geq 2$ and $d_1,\dots,d_m\in\N$. For each
$i\in\{1,\dots,m\}$, consider on the lattice $\Z^{d_i}$ a probability measure $\mu_i$ with $\mathrm{supp}(\mu_i)=\{\pm e_j^{(i)} \mid 1\leq j\leq d_i\}$, where  $e_j^{(i)}$ is the $j$-th unit vector in $\Z^{d_i}$. For any $\alpha_1,\dots,\alpha_m>0$ with $\sum_{i=1}^m\alpha_i=1$, let $\mu:=\sum_{i=1}^m \alpha_i\mu_i$ govern a 
(irreducible) random walk on the free product $\Z^{d_1}\ast\dots \ast\Z^{d_m}$
starting at $e$, where $e$ denotes the identity of the free product. 
\par
Then the return probabilities $\mu^{(2n)}(e)$ behave asymptotically either like $C\cdot \varrho^{2n} \cdot n^{-d_i/2}$ for  $i\in\{1,\dots,m\}$ or like  $C\cdot \varrho^{2n} \cdot n^{-3/2}$ for some constant $C=C_\mu$ depending on $\mu$. 
Moreover, if all exponents $d_i$ are different and $\min\{d_1,\dots,d_m\}\geq 5$ then exactly $m+1$ different asymptotic behaviours may occur by choosing the random walk adequately.
\end{Th}
We will consider more general free products which go beyond free products of lattices. For this purpose, we will present a new approach in order to be able to deal with irreducible
random walks on any free product of the form $\Gamma_1 \ast \ldots \ast
\Gamma_m$. At this point we have to make the following assumption: if the Green function of the random walk on the free factor $\Gamma_i$ is differentiable at its radius of convergence $\r_i$, then the Green function is assumed to have a singular expansion (i.e. in a neighbourhood of $\r_i$) containing only singular terms of the form  $(\r_i-z)^q \log^k(\r_i-z)$ with $q\in (1,\infty )$ and $k\in\N_0$ up to sufficiently large order. The latter property is satisfied for several
well-known groups as e.g. $\Z^d$ or $\Z^d\times (\Z/n\Z)$ with $d\geq 5$ and $n\geq 2 $. If, however, the Green function of the random walk on the free factor $\Gamma_i$ is \textit{not} differentiable at $\r_i$, we do \textit{not} need any assumption on the expansions.
\par
If the asymptotic $n$-step return probabilities of the
random walk on $\Gamma_i$ satisfy  a law of the form $\r_i^{-n\delta}n^{-\lambda_i}
\log^{\kappa_i}n$ then we
will show that only one of the following non-exponential types may occur for
the random walk on the free product: either $n^{-\lambda_i}\log^{\kappa_i}n$ for some
$i\in\{1,\dots,m\}$, or $n^{-3/2}$. That is, we may have
up to $m+1$ different types of asymptotic behaviour for (symmetric or non-symmetric) random
walks, and Theorem \ref{thm:1.1} shows that one can have indeed exactly $m+1$
different behaviours. Moreover, for the case $\Gamma_1\ast \Gamma_2$ equipped
with the probability measure $\mu=\alpha_1 \mu_1 + (1-\alpha_1)\mu_2$, where
$\mu_1$ and $\mu_2$ are probability measures on $\Gamma_1$ and $\Gamma_2$ and
$\alpha_1\in (0,1)$, we characterize the phase transitions of the
non-exponential types in terms of $\alpha_1$. We split the $(0,1)$-interval,
i.e. the interval of possible values for $\alpha_1$, in up to three distinct subintervals such that, in each of them, we have exactly one of the non-exponential types $n^{-\lambda_1}\log^{\kappa_1}(n)$,
$n^{-\lambda_2}\log^{\kappa_2}n$ or $n^{-3/2}$.
\par
Let us briefly recall some results regarding the asymptotic behaviour of return
probabilities. Work in this direction has been done since the 1970's by Gerl,
Sawyer, Woess, Cartwright, Soardi and Lalley, see e.g. 
\cite{gerl-woess}, \cite{sawyer78},  \cite{woess3}, \cite{cartwright-soardi},
\cite{lalley93}. 
Sawyer \cite{sawyer78} applies Fourier analysis to isotropic
random walks on trees (free groups), which uses in a crucial way methods from
complex analysis.
For finite range random walks on free groups, it is known from
\cite{gerl-woess} and \cite{lalley93} that the $n$-step return probabilities
behave asymptotically like $C\varrho^{n}n^{-3/2}$, where $\varrho<1$.
In \cite{gerl2}, \cite{woess82} and \cite{woess3} free products
of finite groups are considered, which have a very tree-like structure and where random walks
obey a $n^{-3/2}$-law. In the more general case
of free products of arbitrary groups the interior structure of each free factor
is more complicated. Woess \cite{woess3}, Cartwright and Soardi
\cite{cartwright-soardi}, Voiculescu \cite{voiculescu} and McLaughlin \cite{mclaughlin} found
independently a method to determine the Green function of a free product in
terms of functional equations of the Green functions defined on the free factors.
We will study these equations carefully, in order to obtain  -- with the help of the
well-known \textit{method of Darboux} -- the asymptotic behaviour of
the power series' coefficients, which are the sought return probabilities. We refer also to the survey of Woess \cite{woess03}, which
outlines the use of generating functions. 
More recently, random walks on free products have also been studied by Mairesse
and Math\'eus \cite{mairesse1} and Gilch \cite{gilch:07}, \cite{gilch:09}, regarding boundary
theory, entropy and rate of escape. For more details and references we refer to
Woess \cite{woess}, which serves also as reference text for our work.
\par
The structure of this paper is as follows: in Section \ref{sec:RW} we introduce some
basic facts and notations. In Section \ref{sec:mainresults} we prove our main
result for the case $\Gamma_1\ast \Gamma_2$, while in
\mbox{Sections \ref{sec:psi=0}} and \ref{sec:lowdim} we are completing the list of degenerate cases, which, in particular, may occur if the Green functions of the random walks on the single factors are non-differentiable at their radii of convergence.
In Section \ref{sec:higher-order} we are proving inductively the proposed asymptotic behaviour
for multi-factor free products of the form $\Gamma_1\ast \ldots \ast\Gamma_m$
with $m\geq 3$. Section \ref{sec:exmaples} discusses some examples. This
includes the case of free
products of the form $\Z^{d_1}\ast \ldots \ast \Z^{d_m}$, where we give a full
classification of the asymptotic behaviour of the return probabilities, which proves Theorem \ref{thm:1.1}.
For $\Gamma_1\ast\Gamma_2$, we give in Section \ref{sec:phase-transition} a
full characterization of the possible phase transition behaviour of the
non-exponential types of the return probabilities in terms of the weight $\alpha_1 $ of the
probability measure given on $\Gamma_1$.
Finally, \mbox{Section \ref{sec:remarks}} gives some concluding remarks about higher asymptotic order terms.

\section{Random Walks on Free Products}\label{sec:RW}

Let $m\in\N$ with $m\geq 2$. Suppose we are given finitely generated groups
$\Gamma_1,\dots,\Gamma_m$, and we denote by $e_i$ the identity of $\Gamma_i$. We consider the \textit{free product} $\Gamma:=\Gamma_1\ast \ldots \ast
\Gamma_m$, which consists of all
finite words of the form
\begin{equation}\label{nr}
x_1x_2 \dots x_n,
\end{equation}
where $x_1,\dots,x_n\in \bigcup_{i=1}^m \Gamma_i\setminus \{e_i\}$
and two consecutive letters are not from the same free factor $\Gamma_i$. In the case $\Gamma_i=\Gamma_j$ we may think that the elements
of $\Gamma_i$ and $\Gamma_j$ have different colours to distinguish their origin. Observe that each factor $\Gamma_i$ can be naturally embedded into $\Gamma$, and therefore $e_i\in\Gamma_i$ can be identified with the empty word $e\in\Gamma$. The free product is a group with $e$ as
identity: the product of two elements is given by
concatenation followed by iterated contractions and cancellations of redundant terms in the middle, in order to
obtain the requested form (\ref{nr}). For example, if $a,b\in \Gamma_1\setminus \{e_1\}$ and $c\in \Gamma_2\setminus \{e_2\}$, such that $c^2\neq e $, then $(aca^{-1})(aca)=ac^2a $. For further details about free products see e.g. Lyndon and Schupp \cite{lyndon-schupp}. 
\par
We recall and introduce some notation: for any function
$f:D\subseteq \mathbb{C}\to\mathbb{C}$ with \mbox{$f(z_0)=0$} for $z_0\in D$, $0<q\in\mathbb{R}$ and $k\in\N_0$, we use the notation $f(z)=\mathbf{o}\bigr((z_0-z)^{q} \log^{k} (z_0-z)\bigl)$,
\mbox{$f(z) = \O_c\bigr((z_0-z)^{q} \log^{k} (z_0-z)\bigl)$}
or $f(z)  =  \O\bigr((z_0-z)^{q} \log^{k}
(z_0-z)\bigl)$,
if for $z\to z_0$ the function $f(z)$ divided by $(z_0-z)^{q} \log^{k} (z_0-z)$ tends to zero, has
a non-zero finite limit or is bounded nearby $z_0$ (that is, the quotient has a finite limes superior) respectively. Furthermore, we
write $(z_0-z)^{q_1} \log^{k_1}(z_0-z) \preceq  (z_0-z)^{q_2} \log^{k_2}(z_0-z)$ if and only if
$(z_0-z)^{q_2} \log^{k_2}(z_0-z) = \O\bigl((z_0-z)^{q_1}
\log^{k_1}(z_0-z)\bigr)$. The value $z_0$ will be obvious from the context.
\par
Suppose we are given probability measures
$\mu_i$ on $\Gamma_i$ with $\langle\mathrm{supp}(\mu_i)\rangle=\Gamma_i$ for
each \mbox{$i\in\{1,\dots,m\}$.} These measures $\mu_i$ govern random walks on
$\Gamma_i$, that is, the single step transition probabilities are given by
$p_i(x_i,y_i)=\mu_i(x_i^{-1}y_i)$ for all $x_i,y_i\in\Gamma_i$. (Let us remark that, in the case of $\Gamma_i=\Z^{d_i}$ with $d_i\in\N$, we speak of a nearest neighbour random walk if $\mathrm{supp}(\mu_i)=\{\pm e_j \mid 1\leq j\leq d_i \}$, where $e_j$ is the $j$-th unit vector in $\Z^{d_i}$.)
We lift now $\mu_i$ to a probability measure $\bar \mu_i$ on $\Gamma$ by defining $\bar \mu_i(x):=\mu_i(x)$ if $x\in\Gamma_i$; otherwise we set $\bar \mu_i(x):=0$.
Let $\alpha_1,\dots,\alpha_m>0$ such that $\sum_{i=1}^m \alpha_i=1$. Consider now the probability measure $\mu:=\sum_{i=1}^m \alpha_i\bar \mu_i$ on the free product $\Gamma$, which arises as a convex combination of the $\bar\mu_i$'s. Then the single step transition probabilities on $\Gamma$ given by $p(x,y):=\mu(x^{-1}y)$ for $x,y\in\Gamma$ define a random walk on $\Gamma$, which is an irreducible Markov chain. We denote by $\mu^{(n)}_1, \dots, \mu^{(n)}_m$ and $\mu^{(n)}$ the $n$-fold convolution
power of $\mu_1, \dots, \mu_m$ and $\mu$, that is, the distribution after $n$ steps
with start at the identity. For
$z\in\mathbb{C}$, the associated \textit{Green functions}
of the random walks on $\Gamma_i$ and $\Gamma$ are given by
\[
G_i(z) :=\sum_{n=0}^{\infty}\mu_i^{(n)}(e_i) z^n\quad \textrm{ and } \quad
G(z)  := \sum_{n=0}^{\infty}\mu^{(n)}(e) z^n.
\]
The corresponding radii of convergence are denoted by $\r_i$ and
$\r$ respectively, which are singularities according to Pringsheim's Theorem. Note that $\r>1$, since $\Gamma$ is non-amenable unless
\mbox{$\Gamma=(\Z/2\Z)\ast (\Z/2\Z)$} (see
e.g. \cite[Theorem 10.10, Corollary 12.5]{woess}; in the latter case the random walk on $\Gamma$ is recurrent). In the following we assume that $G_i(z)$ is exactly $d_i$-times differentiable at $z=\r_i$, where $d_i\in\N_0$. At this point we make the
\textit{basic assumption} that -- whenever $G_i'(\r_i)<\infty$ -- the expansions
of the Green functions $G_i(z)$ in a neighbourhood of $z=\r_i$ have the form
\begin{equation}\label{equ:expansion-assumption}
G_i(z)= \sum_{k=0}^{d_i} g_k^{(i)} (\r_i-z)^k + \sum_{(q,k)\in \mathcal{T}_i}
g^{(i)}_{(q,k)} (\r_i-z)^q \log^k(\r_i-z) + \O\bigl((\r_i-z)^{d_i+2}\bigr),
\end{equation}
where  $\mathcal{T}_i$ is a finite subset of $\{(q,k)\in\mathbb{R}\times \N_0
\mid d_i < q \leq d_i+2\}$.
In other words, the expansions contain only logarithmic and algebraic terms as singular terms up to order \mbox{$(\r_i-z)^{d_i+2}$}. As we will see, higher order terms are not necessary for  the computation of the non-exponential type of the $n$-step return probabilities of the random walk on $\Gamma$. In the following we want to motivate this assumption on $G_i(z)$. This property for the expansion is satisfied in several well-known
cases: for example, the Green functions of nearest neighbour random walks on
lattices $\Z^d$ have such an expansion; see Proposition
\ref{prop:expansion-Gd}. With some effort, such an expansion can be deduced for $\Z^d\times (\Z/n\Z)$ via the same methods  used for $\Z^d$.
In particular, we will prove our main result by induction on the number $m$ of
free factors of $\Gamma$: we will see that the assumptions stated in
(\ref{equ:expansion-assumption}) are stable under free products (except for
some degenerate cases), that is, $G(z)$ has again a similar expansion if
$G'(\r)<\infty$ holds. 
If the Green function $G_i(z)$ has the form (\ref{equ:expansion-assumption}) then the well-known
\textit{method of Darboux} yields that the $n$-step return
probabilities of the random walk on $\Gamma_i$ (governed by $\mu_i$) behave asymptotically like the coefficients of the Taylor expansion of the leading
singular term in (\ref{equ:expansion-assumption}) in a neighbourhood of
$0$. Assume that $S_i(z):=(\r_i-z)^{q_i} \log^{k_i}(\r_i-z)$ is the
\textit{smallest} (or \textit{leading}) singular term in
(\ref{equ:expansion-assumption}) w.r.t. $\preceq$, that is, $q>q_i$ or $(q=q_i
\land k<k_i)$ for all $(q,k)\in\mathcal{T}_i\setminus \{ (q_i,k_i)\}$; then the
coefficients of its expansion in a neighbourhood of $0$ behave asymptotically
like the $n$-step return probabilities on $\Gamma_i$ (the proof of this fact is
completely analogous to the one of Theorem \ref{thm:Psi<0}). More precisely,
they behave like $\hat C_i \r_i^{-n\delta_i} n^{-\lambda_i}\log^{\kappa_i}(n)$, where \mbox{$\delta_i:=\gcd\bigl\lbrace
  n\in\N \mid \mu_i^{(n)}(e_i)>0\bigr\rbrace$} is the period of the random walk on $\Gamma_i$ and
\begin{equation}\label{equ:lambda-kappa}
\lambda_i:=q_i+1\ \textrm{  and }  \kappa_i:=
\begin{cases}
k_i,& \textrm{if } q_i\notin\N,\\
k_i-1 & \textrm{if } q_i\in\N ;
\end{cases}
\end{equation}
see e.g. Flajolet and Sedgewick \cite[Chapter VI.2]{flajolet07} for the asymptotic behaviour of the coefficients in the expansion of $(\r_i-z)^{q_i} \log^{k_i}(\r_i-z)$ in a neighbourhood of $0$.
Analogously, $\delta:=\gcd\bigl\lbrace n\in\N
\mid \mu^{(n)}(e)>0\bigr\rbrace=\gcd\{\delta_1,\dots,\delta_m\}$ is the period of the random walk on $\Gamma$. 
Note that the method of Darboux needs some differentiability assumptions at $z=\r_i$; therefore, we need the expansions of $G_i(z)$ up to terms of order $(\r_i-z)^{d_i+2}$.  For more details about Darboux's method we refer to the comments in the proof of Theorem \ref{thm:Psi<0}. 
We remark that another -- modern -- tool to
handle singular expansions as in (\ref{equ:expansion-assumption}) is
\textit{Singularity Analysis}, which was developed by Flajolet and Odlyzko
\cite{flajolet86}. However, in our context it turns out that the verification
of the specific requirements of singularity analysis is quite cumbersome as one
can also see in Lalley \cite{lalley2}. Let us also point out that, in the case $G_i'(\r_i)=\infty$, we do \textit{not} need any assumptions on the singularity type at $z=\r_i$.
\par
In the following we look at free products of the form $\Gamma_1\ast \Gamma_2$
different from $(\Z/2\Z)\ast (\Z/2\Z)$ (it is well-known that
random walks -- in our context -- on $(\Z/2\Z)\ast (\Z/2\Z)$ obey a $n^{-1/2}$-law). Free products
with more than two factors are discussed in Section \ref{sec:higher-order}.
We introduce the following \textit{first visit generating functions} for $z\in\mathbb{C}$,
$i\in\{1,2\}$ and all $s_i\in\mathrm{supp}(\mu_i)$, $s\in\mathrm{supp}(\mu)=\mathrm{supp}(\mu_1)\cup \mathrm{supp}(\mu_2)$:
\begin{eqnarray*}
F_i\bigl(s_i \bigl| z\bigr) & := & \sum_{n\geq 0} 
\Prob\bigl[ X_n^{(i)} = e_i, \forall m<n: X_m^{(i)}\neq e_i\, \bigl|\,
X_0^{(i)} = s_i\bigr]\,z^n,\\
F\bigl(s \bigl| z\bigr) & := & \sum_{n\geq 0} 
\Prob\bigl[ X_n = e, \forall m<n: X_m\neq e\, \bigl|\,
X_0 = s\bigr]\,z^n,
\end{eqnarray*}
where $\bigl(X_n^{(i)}\bigr)_{n\in\N_0}$ is a random walk on $\Gamma_i$ governed by
$\mu_i$. By conditioning on the number of visits of $e_i$ the functions $F_i\bigl(s_i \bigl| z\bigr)$ are directly linked with $G_i(z)$ via
\begin{equation}\label{equ:G-i}
G_i(z) = \frac{1}{1-\sum_{s_i\in\mathrm{supp}(\mu_i)}\mu_i(s_i)\,z\, F_i\bigl(s_i \bigl| z\bigr)}.
\end{equation}
In the following we will summarize some further important basic facts, where we will refer to Woess
\cite{woess} for further details.
Define 
\begin{equation}\label{def:xi}
\begin{array}{rcl}
\zeta_1(z) & := & \frac{\alpha_1 z}{1-\alpha_2 z \sum_{s_2\in\mathrm{supp}(\mu_2)} \mu_2(s_2)
  F(s_2| z)} \ \textrm{ and } \\[3ex]
\zeta_2(z) & := &\frac{\alpha_2 z}{1-\alpha_1 z \sum_{s_1\in\mathrm{supp}(\mu_1)}
  \mu_1(s_1)   F(s_1| z)}.
\end{array}
\end{equation}
Note that $\zeta_i(1)$ is the probability of starting at $e$ and making a step
from $e$ w.r.t. $\mu_i$ after finite time. Observe that $F\bigl(s_i \bigl| z\bigr)=F_i\bigl(s_i \bigl| \zeta_i(z)\bigr)$ for $s_i\in\mathrm{supp}(\mu_i)$; see \cite[Proposition 9.18c)]{woess}. By \cite[Equation (9.20)]{woess} and (\ref{equ:G-i}), the functions $F_i\bigl(s_i \bigl| \zeta_i(z)\bigr)$, $G_i(z)$ and $G(z)$ can be linked as follows:
\begin{equation}\label{equ:G-Gi-link}
G(z)= \frac{\zeta_i(z)}{\alpha_i\,z} G_i\bigl(\zeta_i(z)\bigr)=\frac{\zeta_i(z)}{\alpha_i\,z\,\Bigl(1-\sum_{s_i\in\mathrm{supp}(\mu_i)}\mu_i(s_i)\,\zeta_i(z)\, F_i\bigl(s_i \bigl| \zeta_i(z)\bigr)\Bigr)}.
\end{equation}
Hence, our aim will be to determine an expansion of $\zeta_i(z)$ in a neighbourhood of $z=\r$, in order to get a singular expansion for $G(z)$ in a neighbourhood of $z=\r$.
By \cite[Proposition 9.10]{woess}, there are functions $\Phi_i$, $i\in\{1,2\}$, and $\Phi$ such that
\begin{equation}\label{phi_12}
G_i(z) = \Phi_i\bigl(zG_i(z)\bigr) \ \textrm{ and } \
G(z) = \Phi\bigl(zG(z)\bigr)
\end{equation}
for all $z\in\mathbb{C}$ in an open neighbourhood of the intervals $[0,\r_i)$ and $[0,\r)$ respectively.
In particular, the functions $\Phi_i$ and $\Phi$ are analytic in an
open neighbourhood of the intervals $[0,\theta_i)$ and
$[0,\theta)$ respectively, where $\theta_i:=\r_iG_i(\r_i)$ and
$\theta:=\r G(\r)$. $\Phi_i$ and $\Phi$ are also strictly increasing and strictly convex in $[0,\theta_i)$ and
$[0,\theta)$ respectively. Furthermore, we define 
\begin{equation}\label{equ:Phi-Psi}
\Psi_i(t):=\Phi_i(t)-t\Phi_i'(t) \quad \textrm{ and } \quad
\Psi(t):=\Phi(t)-t\Phi'(t).
\end{equation}
By \cite[Theorem 9.19]{woess}, 
\begin{equation}\label{equ:Phi-Psi-formula}
\Phi(t) = \Phi_1(\alpha_1 t) + \Phi_2(\alpha_2
t)-1 \quad \textrm{ and } \quad
\Psi(t) = \Psi_1(\alpha_1 t) + \Psi_2(\alpha_2 t)-1.
\end{equation}
We write $\Psi_i(\theta_i):=\lim_{t\to\theta_i-} \Psi_i(t)$.
Define
$$
\bar{\theta}:=\min \left \{
  \frac{\theta_1}{\alpha_1},\frac{\theta_2}{\alpha_2}\right \}.
$$
We will make a case distinction according to finiteness of $G_i(\r_i)$ and $G_i'(\r_i)$
and also to the sign
of $\Psi(\bar\theta):=\lim_{t\to\bar\theta-} \Psi(t)$. If $\Psi(\bar\theta)<0$ then the $n$-step return
probabilities of the random walk on $\Gamma$ behave asymptotically like
$$
\mu^{(n\delta)}(e) \sim C \cdot \r^{-n\delta} \cdot n^{-3/2}
$$ 
and the Green function of the random walk on $\Gamma$ has the form
\begin{equation}\label{equ:G-Psi<0}
G(z)=A(z)+\sqrt{\r-z} B(z),
\end{equation}
where $A(z),B(z)$ are analytic functions in a neighbourhood of $z=\r$ with
$B(\r)\neq 0$; see \cite[Theorem 17.3]{woess} or \cite[Section
VI.7.]{flajolet07}. Moreover, if one fixes any finite, symmetric sets
$\mathcal{S}_i$ of
generators of $\Gamma_i$ for $i\in\{1,2\}$, where each $\mathcal{S}_i$ contains at least one element
of order bigger than $2$, then $\mu_1$, $\mu_2$ and $\alpha_1$ can always be
chosen in a suitable way in order to obtain $\Psi(\bar\theta)<0$ with $\mathrm{supp}(\mu_i)=\mathcal{S}_i$; see \cite[Corollary 17.10]{woess}.
In particular, the same asymptotic behaviour (including an expansion of the Green function of the form (\ref{equ:G-Psi<0})) holds if $\Gamma_1$ and $\Gamma_2$ are finite, see \cite{woess3}. Therefore, we assume from
now on that at least one out of $\Gamma_1$ and $\Gamma_2$ is infinite, and we may restrict our
investigation to the cases $\Psi(\bar\theta)>0$ and $\Psi(\bar\theta)=0$. 
\par 
We remark some important facts for the case $\Psi(\bar\theta)\geq 0$. If the latter holds, we have $\theta=\bar\theta$ and $G(\r)<\infty$, see \cite[Theorem 9.22]{woess}. By \cite[Lemma 17.1.a)]{woess}, we have $\zeta_i(\r)\leq \r_i$ for $i\in\{1,2\}$ with equality if and only if $\theta =\theta_i/\alpha_i$.
\par
The proof for the asymptotic behaviours of the return probabilities is split up over the following sections. In Section \ref{sec:mainresults} we calculate the asymptotics in the case when $\Psi(\bar\theta) >0$, $G_1'(\r_1)<\infty$ and $G_2'(\r_2)<\infty$ hold; see Theorem \ref{thm:Psi<0}. In Section \ref{sec:psi=0} we investigate the case  when $\Psi(\bar\theta)=0$, $G_1'(\r_1)<\infty$ and $G_2'(\r_2)<\infty$ hold; see Theorem \ref{th:psi-bar-theta=0}. From the proof of this theorem we will see that even the case $\Psi(\bar\theta)=0$, $G_1'\bigl(\zeta_1(\r_1)\bigr)<\infty$ and $G_2'\bigl(\zeta_2(\r_2)\bigr)<\infty$ is covered. In Section \ref{sec:lowdim} we treat the remaining cases: Theorem \ref{thm:5.1} covers the case when $G_1(\r_1)<\infty$, $G_1'(\r_1)=\infty$ and $G_2'(\r_2)<\infty$ hold, while Corollary \ref{Cor:5.2} answers the question for the asymptotic behaviour when $G_1'(\r_1)=\infty$ and $G_2'(\r_2)=\infty$. Finally, Theorem \ref{thm:5.3} covers the remaining case when $G_1(\r_1)=\infty$ or $G_2(\r_2)=\infty$.

\section{The Asymptotic Behaviour in the Case $\Psi(\bar\theta) >0$}
\label{sec:mainresults}
Throughout this section we investigate the case $m=2$ and assume that $\Psi(\bar\theta)>0$ and $G_1(z)$
and $G_2(z)$ are differentiable at their radii of
convergence. That is, the Green functions have an expansion as assumed in (\ref{equ:expansion-assumption}). Recall that the smallest singular term w.r.t. $\preceq$ in the expansion of $G_i(z)$, is denoted by $S_i(z)=(\r_i-z)^{q_i}\log^{k_i}(\r_i-z)$ with $d_i<q_i\leq
d_i+1$. Let us remark that Darboux's method yields that the $n$-step return probabilities of the random walk on $\Gamma_i$ governed by $\mu_i$ behave asymptotically
like $\hat C_i \r_i^{-n\delta_i} n^{-\lambda_i}\log^{\kappa_i} n$,
where $\lambda_i$ and $\kappa_i$ are given by (\ref{equ:lambda-kappa}). The aim of this section is to prove the following:
\begin{Th}\label{thm:Psi<0}
Assume that $G_1(z)$ and $G_2(z)$ are differentiable at $z=\r_1$, $z=\r_2$
respectively, and have an expansion as in (\ref{equ:expansion-assumption}). If $S_1(z)\preceq S_2(z)$ and $\Psi(\bar \theta)>0$ then:
$$
\mu^{(n\delta)}(e) \sim 
\begin{cases}
C_1\cdot \r^{-n\delta}\cdot  n^{-\lambda_1}\cdot \log^{\kappa_1}(n), & \textrm{if } \alpha_1 \geq \frac{\theta_1}{\theta_1+\theta_2},\\
C_2\cdot \r^{-n\delta}\cdot  n^{-\lambda_2}\cdot \log^{\kappa_2}(n), & \textrm{if } \alpha_1 < \frac{\theta_1}{\theta_1+\theta_2},
\end{cases}
\quad \textrm{for some constants } C_1, C_2>0.
$$
\end{Th}
In the following we may assume w.l.o.g. that $\theta = \bar\theta =
\theta_1/\alpha_1$.
Recall that $F(s_i | z)= F_i(s_i |\zeta_i (z))$ for all $s_i\in \operatorname{supp}(\mu_i) $. Then we rewrite (\ref{def:xi}) as follows:
\begin{eqnarray}
\alpha_1 z & = & \zeta_1(z) \Bigl( 1- \alpha_2 z \sum_{s_2\in\mathrm{supp}(\mu_2)} \mu_2(s_2)
F_2\bigl(s_2|\zeta_2(z)\bigr)\Bigr) ,\label{equ:xi1}\\
\alpha_2 z & = & \zeta_2(z) \Bigl( 1- \alpha_1 z \sum_{s_1\in\mathrm{supp}(\mu_1)} \mu_1(s_1)
F_1\bigl(s_1|\zeta_1(z)\bigr)\Bigr) .\label{equ:xi2}
\end{eqnarray}
Recall that $\zeta_1(\r)=\r_1$ and $\zeta_2(\r) \leq \r_2$ with equality if and only if $\theta =\theta_1/\alpha_1=\theta_2/\alpha_2$. We remark also that $\Psi(\bar\theta) >0$ implies $G'(\r)<\infty$: since $\Phi'(\bar\theta)<\Phi(\bar\theta)/\bar\theta=1/\r$ we get by differentiating (\ref{phi_12})
$$
G'(\r)=\lim_{z\to\r} \frac{\Phi'\bigl(zG(z)\bigr)\,G(z)}{1-z\,\Phi'\bigl(zG(z)\bigr)}=\frac{\Phi'(\bar\theta)\,G(\r)}{1-\r\,\Phi'(\bar\theta)}<\infty.
$$
Furthermore, we define 
$$
D:=
\begin{cases}
d_1, &\textrm{if } \bar\theta < \theta_2/\alpha_2,\\
\min\{d_1,d_2\}, & \textrm{if } \bar\theta =\theta_1/\alpha_1= \theta_2/\alpha_2.
\end{cases}
$$
We denote by $S(z)$ the \textit{main leading singular term}, which is given by 
$$
S(z)=
\begin{cases}
S_1(z),& \textrm{if } \bar\theta < \theta_2/\alpha_2,\\
\min\bigl\lbrace S_1(z),S_2(z)\bigr\rbrace,& \textrm{if } \bar\theta = \theta_2/\alpha_2.
\end{cases}
$$
\begin{Lemma}\label{lemma:xi_derivatives}
$0<\zeta_1'(\r)<\infty$ and  $0<\zeta_2'(\r)<\infty$.
\end{Lemma}
\begin{proof}
We prove the lemma only for $\zeta_1'(\r)$. We write
$$
H_2(z) := \alpha_2 z \sum_{s_2\in\mathrm{supp}(\mu_2)} \mu_2(s_2)
F_2\bigl(s_2|\zeta_2(z)\bigr).
$$
Since $\zeta_1(\r)= \r_1$, we have $H_2(\r)<1$; compare with the definition of $\zeta_1(z)$.  Furthermore, the coefficient of $z^n$ in $H_2(z)$ is just the probability for the random walk on $\Gamma$ of starting at $e$, making the first step w.r.t. $\mu_2$ and returning for the first time to $e$ at time $n$. Thus, this probability is bounded from above by $\mu^{(n)}(e)$, and consequently $H_2'(\r)< G'(\r)<\infty$. Computing the derivative of $\zeta_1(z)$ in a neighbourhood of $z=\r$ gives
\[
 \zeta_1'(z)=\frac{\alpha_1\bigl(1-H_2(z)\bigr)+\alpha_1 z H_2'(z)}{\bigl(1-H_2(z)\bigr)^2}>0 .
\]
Finiteness of $\zeta_1'(\r)$ follows now directly from the remarks above.
\end{proof}
The functions $F_i(s_i|z)$, where $i\in\{1,2\}$ and
$s_i\in\mathrm{supp}(\mu_i)$, are at least $d_i$-times differentiable at $z=\r_i$, since the same holds for $G_i(z)$ and we can compare the coefficients of $z^n$ in the definitions of $F_i(s_i|z)$ and $G_i(z)$ as follows:
$$
\mu_i^{(n)}(e_i) \geq \mu_i(s_i)\cdot \Prob\bigl[ X_n^{(i)} = e_i, \forall m<n: X_m^{(i)}\neq e_i\, \bigl|\,
X_0^{(i)} = s_i\bigr].
$$
Thus, we can rewrite these functions in the form
\begin{equation}\label{equ:F-expansion}
F_i(s_i|z) = \sum_{n=0}^{d_i} f_n(s_i) (\r_i-z)^n + E^{(i)}(s_i|z)
\end{equation}
with coefficients $f_n(s_i)\in\mathbb{R}$ and $E^{(i)}(s_i|z)=\mathbf{o}\bigl((\r_i-z)^{d_i}\bigr)$. If $\zeta_2(\r)<\r_2$ then $F_2(s_2|z)$ is analytic at $z=\zeta_2(\r)$ for all $s_2\in\mathrm{supp}(\mu_2)$ and we can even write 
$$
F_2(s_2|z) = \sum_{n\geq 0} f_n(s_2) \bigl(\zeta_2(\r)-z\bigr)^n.
$$
Now we can prove:
\begin{Lemma}\label{lemma:F-expansion-rest}
For $z\in\mathbb{C}$ in a neighbourhood of $\r_i$,
\begin{eqnarray*}
&&\sum_{s_i\in\mathrm{supp}(\mu_i)} \mu_i(s_i)\,z\, E^{(i)}(s_i|z) \\
&=& 
e^{(i)}_{(q_i,k_i)} (\r_i-z)^{q_i} \log^{k_i}(\r_i-z) +
\sum_{(q,k)\in\widehat{\mathcal{T}_i}} e^{(i)}_{(q,k)} (\r_i-z)^q \log^k(\r_i-z) +
\O\bigl((\r_i-z)^{d_i+2}\bigr),
\end{eqnarray*}
where $e^{(i)}_{(q_i,k_i)}\neq 0$ and $\widehat{\mathcal{T}_i}\subseteq \bigl\lbrace (q,k)\in\mathbb{R}\times \N_0 \mid d_i<q\leq d_i+2,
q>q_i \textrm{ or } (q=q_i \Rightarrow k<k_i)\bigr\rbrace$ is finite.
\end{Lemma}
\begin{proof}
Define
$$
U_i(z):=\sum_{s_i\in\mathrm{supp}(\mu_i)}\mu_i(s_i)\,z\,F_i(s_i| z).
$$
Observe that the expansions of $U_i(z)$ and $G_i(z)$ have the same leading
singular term:
indeed, both functions are $d_i$-times differentiable in a
neighbourhood of $z=\r_i$ due to the well-known equation $G_i(z)=1/\bigl(1-U_i(z)\bigr)$. Therefore, we have expansions
\[
 G_i(z)=\sum_{k=0}^{d_i}g_k^{(i)}(\r_i-z)^k+R_{G_i}(z)\quad \textnormal{and}\quad
U_i(z)=\sum_{k=0}^{d_i}u_k^{(i)}(\r_i-z)^k+R_{U_i}(z),
\]
where $R_{G_i}(z)=\O_c\bigl(S_i(z)\bigr)$ and $R_{U_i}(z)=\o\bigl((\r_i-z)^{d_i}\bigr)$.
Substituting these expansions into $G_i(z) (1-U_i(z))=1$, and taking all
polynomial terms to one side, we get
$$
\bigl(1-U_i(\r_i)\bigr)\, R_{G_i}(z) - G_i(\r_i)\, R_{U_i(z)} = p(z) + \o\bigl((\r_i-z)^{d_i+1}\bigr),
$$
where $p(z)$ is some polynomial. This equation implies that the right hand side is of order $\O\bigl((\r_i-z)^{d_i+1}\bigr)$, that is, $R_{U_i(z)}=\O_c\bigl(S_i(z)\bigr)$ and
we can write
$$
U_i(z)=\sum_{k=0}^{d_i}u_k^{(i)}(\r_i-z)^k+u_{(q_i,k_i)}^{(i)} S_i(z) +
\widehat{R}_{U_i}(z)\quad \textrm{ with } \widehat{R}_{U_i}(z)=\o\bigl(S_i(z)\bigr).
$$
Plugging this expansion once again into $G_i(z) (1-U_i(z))=1$, comparing
error terms and iterating the last steps, together with substituting (\ref{equ:F-expansion}) in the definition of $U_i(z)$, yields the claim.
\end{proof}
The next goal is to show that $\zeta_1(z)$ and $\zeta_2(z)$ are $D$-times
differentiable at $z=\r$. 
\begin{Prop}\label{prop:xi-expansion}
There are real numbers $x_0,x_1,\dots,x_{D}$ and $y_0,y_1,\dots,y_{D}$ such that
$$
\zeta_1(z)  = \sum_{k=0}^{D} x_k\,(\r-z)^k + X_{D}^{(1)}(z)\ \textrm{ and }\
\zeta_2(z)  = \sum_{k=0}^{D} y_k\,(\r-z)^k + X_{D}^{(2)}(z),
$$
where $X_{D}^{(1)}(z)=\mathbf{o}\bigl((\r-z)^{D}\bigr)$ and
$X_{D}^{(2)}(z)=\mathbf{o}\bigl((\r-z)^{D}\bigr)$.
\end{Prop}
\begin{proof}
We prove the proposition by determining $x_0,x_1,\dots,x_{D}$ and
$y_0,y_1,\dots,y_{D}$ inductively. By Lemma \ref{lemma:xi_derivatives} and a well-known characterization of differentiability, we can rewrite $\zeta_1(z)$ and $\zeta_2(z)$ in the following way:
\begin{equation}\label{equ:xi-1st}
\begin{array}{rcl}
\zeta_1(z) & = & \r_1 - \zeta_1'(\r)\,(\r-z) + X_1^{(1)}(z), \textrm{ where }
X_1^{(1)}(z) = \mathbf{o}(\r-z),\\[2ex]
\zeta_2(z) & = & \zeta_2(\r) - \zeta_2'(\r)\,(\r-z) + X_1^{(2)}(z), \textrm{ where }
X_1^{(2)}(z) = \mathbf{o}(\r-z).
\end{array}
\end{equation}
Thus, we have determined $x_0, x_1$ and $y_0,y_1$. 
Assume now that we can write for some $t< D$
\begin{equation}\label{equ:1st-xi-expansion}
\zeta_1(z) =  \sum_{k=0}^{t} x_k\,(\r-z)^k + X_{t}^{(1)}(z)\ \textrm{ and }\
\zeta_2(z) =  \sum_{k=0}^{t} y_k\,(\r-z)^k + X_{t}^{(2)}(z),
\end{equation}
where $X_{t}^{(1)}(z)=\mathbf{o}\bigl((\r-z)^{t}\bigr)$ and
$X_{t}^{(2)}(z)=\mathbf{o}\bigl((\r-z)^{t}\bigr)$. 
Recall from (\ref{equ:F-expansion}) that we have expansions of
$F_1(s_1 | z\bigr)$ and $F_2(s_2 | z)$ of the form
\begin{equation}
\label{equ:1st-F-expansion}
\begin{array}{rcl}
F_1(s_1| z)  & = &  \sum_{n=0}^{D} a_n(s_1)(\r_1-z)^n +
E^{(1)}(s_1|z)\ \textrm{ and } \\[2ex]
F_2(s_2 |z) & = & \sum_{n=0}^{D} b_n(s_2)\bigl(
\zeta_2(\r)-z\bigr)^n + E^{(2)}(s_2|z),
\end{array}
\end{equation}
where $E^{(i)}(s_i|z)=\mathbf{o}\bigl((\zeta_i(\r)-z)^{D}\bigr)$. In particular, if $\bar\theta<\theta_2/\alpha_2$ then $\zeta_2(\r)<\r_2$ and consequently we can even write $F_2(s_2 |z)  =  \sum_{n\geq 0} b_n(s_2)\bigl(\zeta_2(\r)-z\bigr)^n$. Recall that the case $D=d_1>d_2$ implies $\bar\theta<\theta_2/\alpha_2$.
We now substitute the expansions
(\ref{equ:1st-xi-expansion}) and (\ref{equ:1st-F-expansion}) in Equations
(\ref{equ:xi1}) and (\ref{equ:xi2}), yielding the following system:
\begin{equation}\label{equ:system-begin}
\begin{array}{rcl}
\alpha_1 z & = & \Bigl(\sum_{k=0}^{t} x_k\,(\r-z)^k + X_{t}^{(1)}(z)\Bigr ) \Biggl[ 1 - \alpha_2 \bigl(\r-(\r-z)\bigr) \sum_{s_2\in\mathrm{supp}(\mu_2)}
\mu_2(s_2) \cdot \\
 && \quad \cdot \Bigl[\sum_{n=0}^{D}
b_n(s_2) \Bigl(-\sum_{k=1}^{t} y_k\,(\r-z)^k - X_{t}^{(2)}(z)\Bigr)^n +
E^{(2)}\bigl(s_2 \bigl| \zeta_2(z)\bigr)\Bigr]\Biggr], \\
\alpha_2 z & = & \Bigl(\sum_{k=0}^{t} y_k\,(\r-z)^k + X_{t}^{(2)}(z)\Bigr ) \Biggl[ 1 - \alpha_1 \bigl(\r-(\r-z)\bigr) \sum_{s_1\in\mathrm{supp}(\mu_1)}
\mu_1(s_1) \cdot \\
&&\quad \cdot \Bigl[\sum_{n=0}^{D}
a_n(s_1) \Bigl(-\sum_{k=1}^{t} x_k\,(\r-z)^k - X_{t}^{(1)}(z)\Bigr)^n
+E^{(1)}\bigl(s_1\bigl| \zeta_1(z)\bigr)\Bigr]\Biggr]. 
\end{array}
\end{equation}
Observe that $\sum_{s_i\in\mathrm{supp}(\mu_i)}\mu(s_i)z E^{(i)}\bigl (s_i\bigl| \zeta_i(z)\bigr) =\mathbf{o}\bigl((\zeta_i(\r)-\zeta_i(z))^{D}\bigr)=\mathbf{o}\bigl((\r-z)^{D}\bigr)$.
We now bring all polynomial and higher order terms to the left hand side and get:
\begin{equation}\label{equ:system0}
\begin{array}{rcl}
 P_t^{(1)} (z) + \mathbf{o}\bigl((\r  - z)^{t+1}\bigr)  
& = &\Bigl[1 - \alpha_2 \r \sum_{s_2\in\mathrm{supp}(\mu_2)}\mu_2(s_2)
b_0(s_2)\Bigr ] X_{t}^{(1)} (z)  \\[1ex]
&&\quad +  \Bigl[\alpha_2 \r_1 \r \sum_{s_2\in\mathrm{supp}(\mu_2)}\mu_2(s_2)
b_1(s_2)\Bigr]X_{t}^{(2)} (z),\\[1ex]
P_t^{(2)} (z)  +  \mathbf{o}\bigl((\r - z)^{t+1}\bigr) 
& = & \Bigl[ \alpha_1 \zeta_2(\r) \r  \sum_{s_1\in\mathrm{supp}(\mu_1)} \mu_1(s_1) a_1 (s_1) \Bigr] X_{t}^{(1)} (z)\!
 \\[1ex]
&&\quad +\Bigl[1 - \alpha_1 \r  \sum_{s_1\in\mathrm{supp}(\mu_1)} \mu_1(s_1) a_0(s_1)\Bigr] X_{t}^{(2)} (z) ,
\end{array}
\end{equation}
where $P_t^{(1)}(z)$ and $P_t^{(2)}(z)$ are polynomials in the variable
$z$. By assumption on $X_{t}^{(1)} (z)$ and $X_{t}^{(2)} (z)$, the right hand sides of (\ref{equ:system0}) are of order $\o\bigl((\r-z)^{t}\bigr)$. Therefore, the left hand sides have to be of order $\O\bigl((\r-z)^{t+1}\bigr)$, and consequently the right hand sides have to be also of order $\O\bigl((\r-z)^{t+1}\bigr)$.
It remains to show that $X_{t}^{(1)}(z)=\O\bigl( (\r-z)^{t+1}\bigr)$ and
$X_{t}^{(2)}(z)=\O\bigl( (\r-z)^{t+1}\bigr)$. For this
purpose, define the matrix $M=(m_{ij})_{1\leq i,j\leq 2}$ by
\begin{eqnarray*}
m_{11} & := & 1- \alpha_2 \r \sum_{s_2\in\mathrm{supp}(\mu_2)}\mu_2(s_2)
b_0(s_2),\\[1ex]
m_{12} & := & \alpha_2 \r_1 \r \sum_{s_2\in\mathrm{supp}(\mu_2)}\mu_2(s_2) b_1(s_2),\\[1ex]
m_{21} & := & \alpha_1 \zeta_2(\r) \r \sum_{s_1\in\mathrm{supp}(\mu_1)}\mu_1(s_1)
a_1(s_1), \\
m_{22} & := & 1- \alpha_1 \r \sum_{s_1\in\mathrm{supp}(\mu_1)}\mu_1(s_1) a_0(s_1).
\end{eqnarray*}
Then the system (\ref{equ:system0}) is equivalent to
$$
M \cdot 
\left(
\begin{array}{c}
X_{t}^{(1)}(z)\\
X_{t}^{(2)}(z)
\end{array}
\right) 
= \left(
\begin{array}{c}
Q_t^{(1)}(z)\\
Q_t^{(2)}(z)
\end{array}
\right),
$$
where $Q_t^{(1)}(z)=\O\bigl((\r-z)^{t+1}\bigr)$ and
$Q_t^{(2)}(z)=\O\bigl((\r-z)^{t+1}\bigr)$. If the matrix $M$ is invertible,
then obviously $X_{t}^{(1)}(z)=\O\bigl((\r-z)^{t+1}\bigr)$ and
$X_{t}^{(2)}(z)=\O\bigl((\r-z)^{t+1}\bigr)$. To this end, we now prove
invertibility of $M$:
\begin{Lemma}
$\det (M) \neq 0$.
\end{Lemma}
\begin{proof}
We start with differentiating equations (\ref{equ:xi1}) and (\ref{equ:xi2}):
\begin{eqnarray*}%\label{equ:system1}
\alpha_1 & = & \Bigl( -\alpha_2 \!\! \sum_{s_2\in\mathrm{supp}(\mu_2)}\!\! \mu_2(s_2)
F_2\bigl(s_2|\zeta_2(z)\bigr) - \alpha_2 z\!\! \sum_{s_2\in\mathrm{supp}(\mu_2)}\!\! \mu_2(s_2)
F_2'\bigl(s_2|\zeta_2(z)\bigr) \zeta_2'(z) \Bigr) \zeta_1(z)\\
&& \quad 
+ \zeta_1'(z) 
\Bigl( 1- \alpha_2 z \sum_{s_2\in\mathrm{supp}(\mu_2)} \mu_2(s_2)
F_2\bigl(s_2|\zeta_2(z)\bigr)\Bigr),\\[2ex]
\alpha_2 & = & \Bigl( - \alpha_1 \!\! \sum_{s_1\in\mathrm{supp}(\mu_1)}\!\! \mu_1(s_1)
F_1\bigl(s_1|\zeta_1(z)\bigr)- \alpha_1 z\!\! \sum_{s_1\in\mathrm{supp}(\mu_1)}\!\! \mu_1(s_1)
F_1'\bigl(s_1|\zeta_1(z)\bigr) \zeta_1'(z)\Bigr) \zeta_2(z) \\
&& \quad + \zeta_2'(z) 
\Bigl( 1- \alpha_1 z \sum_{s_1\in\mathrm{supp}(\mu_1)} \mu_1(s_1)
F_1\bigl(s_1|\zeta_1(z)\bigr)\Bigr). 
\end{eqnarray*}
Observe that we have $a_0(s_1)=F_1(s_1| \r_1)$,
$a_1(s_1)=-F_1'(s_1| \r_1)$,
$b_0(s_2)=F_2\bigl(s_2|\zeta_2(\r)\bigr)$ and
$b_1(s_2)=-F_2'\bigl(s_2|\zeta_2(\r)\bigr)$. Substituting these values in the above system %(\ref{equ:system1})
and letting $z\to \r$ yields
\begin{eqnarray*}%\label{equ:system3}
\alpha_1 & = & \Bigl( -\alpha_2 \sum_{s_2\in\mathrm{supp}(\mu_2)} \mu_2(s_2)
b_0(s_2) + \alpha_2 \r \sum_{s_2\in\mathrm{supp}(\mu_2)} \mu_2(s_2)
b_1(s_2) \zeta_2'(\r) \Bigr) \r_1 
+ \zeta_1'(\r) m_{11},\\
\alpha_2 & = & \Bigl( - \alpha_1\! \sum_{s_1\in\mathrm{supp}(\mu_1)}\! \mu_1(s_1)
a_0(s_1) + \alpha_1 \r \! \sum_{s_1\in\mathrm{supp}(\mu_1)}\! \mu_1(s_1)
a_1(s_1) \zeta_1'(\r)\Bigr) \zeta_2(\r) + \zeta_2'(\r) m_{22}.
\end{eqnarray*}
Since $\zeta_1(\r),\zeta_2(\r)>0$ and $a_1(s_1),b_1(s_2)<0$ the last equations imply
$m_{11},m_{22}>0$.
We proceed with rewriting the last system: % (\ref{equ:system3})
\begin{equation}\label{equ:system5}
\begin{array}{rcl}
\alpha_2 \r_1 \r \sum_{s_2\in\mathrm{supp}(\mu_2)} \mu_2(s_2)
b_1(s_2) \zeta_2'(\r) & = & 
A - \zeta_1'(\r) m_{11},\\[2ex]
\alpha_1  \zeta_2(\r) \r \sum_{s_1\in\mathrm{supp}(\mu_1)} \mu_1(s_1)
a_1(s_1) \zeta_1'(\r) & = & 
B - \zeta_2'(\r) m_{22},
\end{array}
\end{equation}
where
$$
A  :=  \alpha_1 + \alpha_2 \r_1\sum_{s_2\in\mathrm{supp}(\mu_2)} \mu_2(s_2)
b_0(s_2)\ \textrm{ and } \  
B  :=  \alpha_2 + \alpha_1 \zeta_2(\r) \sum_{s_1\in\mathrm{supp}(\mu_1)} \mu_1(s_1)
a_0(s_1).
$$
Multiplying both equations in (\ref{equ:system5}) yields
the equation
$$
\zeta_1'(\r) \zeta_2'(\r)m_{12} m_{21} = AB - \zeta_1'(\r) m_{11} B -\zeta_2'(\r) m_{22} A
+ \zeta_1'(\r) \zeta_2'(\r) m_{11} m_{22}.
$$
Assume now that $\det(M)=0$.  Then we would get
$$
\zeta_1'(\r) m_{11} B + \zeta_2'(\r) m_{22} A = AB,
$$
or equivalently,
\begin{equation}\label{equ:xi2'}
\zeta_2'(\r) = \frac{AB - \zeta_1'(\r) m_{11} B}{m_{22} A}.
\end{equation}
Furthermore, (\ref{equ:system5}) implies
$$
\zeta_1'(\r) = \bigl(A-C \zeta_2'(\r)\bigr) / m_{11},
$$
where $C:=\alpha_2 \r_1 \r \sum_{s_2\in\mathrm{supp}(\mu_2)} \mu_2(s_2)b_1(s_2)<0$.
Substituting the last equation in (\ref{equ:xi2'}) would lead to
$$
\zeta_2'(\r) = \frac{BC}{m_{22}A} \zeta_2'(\r).
$$
Observe now that $A,B,m_{22}>0$ and $C<0$. This yields a contradiction in the
last equation, since $\zeta_2'(\r)>0$. Thus, $\det(M)\neq 0$.
\end{proof}
The last lemma finishes the proof of Proposition \ref{prop:xi-expansion}.
\end{proof}
Recall the definition of the main leading singular term $S(z)=S_i(z)=(\r_i-z)^{q_i} \log^{k_i}(\r_i-z)$.
The next aim is to show that at least one of the functions $X_{D}^{(1)}(z)$ and
$X_{D}^{(2)}(z)$ has order $\O_c\bigl((\r-z)^{q_i} \log^{k_i}
(\r-z)\bigr)$. To this end, we look at the final step of
the induction in the proof of Proposition \ref{prop:xi-expansion}. For $t=D$,
the system (\ref{equ:system-begin}) becomes
\begin{eqnarray*}
&& \Bigl[ 1- \alpha_2 \r \sum_{s_2\in\mathrm{supp}(\mu_2)}\mu_2(s_2) b_0(s_2)\Bigr]\cdot
X_{D}^{(1)}(z) +  \Bigl[ \alpha_2 \r_1 \r
\sum_{s_2\in\mathrm{supp}(\mu_2)}\mu_2(s_2) b_1(s_2) \Bigr] \cdot
X_{D}^{(2)}(z) \\
&& \quad -\alpha_2 \r_1 \sum_{s_2\in\mathrm{supp}(\mu_2)} \mu_2(s_2)\,z\,
E^{(2)}\bigl(s_2|\zeta_2(z)\bigr) = P_{D}^{(1)}(z) + \o\bigl((\r-z)^{D+1}\bigr),\\
&&\Bigl[ \alpha_1 \zeta_2(\r) \r \!\! \sum_{s_1\in\mathrm{supp}(\mu_1)}\!\! \mu_1(s_1) a_1(s_1) \Bigr] \cdot X_{D}^{(1)}(z)
+ \Bigl[ 1- \alpha_1 \r \!\! \sum_{s_1\in\mathrm{supp}(\mu_1)}\!\! \mu_1(s_1) a_0(s_1)\Bigr]\cdot
X_{D}^{(2)}(z) \\
&& \quad -\alpha_1  \zeta_2(\r) \sum_{s_1\in\mathrm{supp}(\mu_1)} \mu_1(s_1)\,z\,
E^{(1)}\bigl(s_1|\zeta_1(z)\bigr) = P_{D}^{(2)}(z) + \o\bigl((\r-z)^{D+1}\bigr),
\end{eqnarray*}
where $P_{D}^{(1)}(z)$ and $P_{D}^{(2)}(z)$ are polynomials in the variable
$z$.  By (\ref{equ:xi-1st}), we may conclude that $\bigl(\zeta_i(\r)-\zeta_i(z)\bigr)=\O_c(\r-z)$. 
Since $\zeta_i'(\r_i)<\infty$ by Lemma \ref{lemma:xi_derivatives}, we have for $1<p\in\mathbb{R}$
$$
\bigl(\zeta_i(\r)-\zeta_i(z)\bigr)^p = \bigl(\zeta_i'(\r_i) (\r-z) + \o(\r-z)\bigr)^p = \zeta_i'(\r_i)^p\,(\r-z)^p\,
\bigl(1+\o(1)\bigr)^p = \O_c\bigl((\r-z)^p\bigr)
$$
and
\begin{eqnarray*}
\log\bigl(\zeta_i(\r)-\zeta_i(z)\bigr) &=& \log\bigl( \zeta_i'(\r_i)(\r-z)+\o(\r-z)\bigr)\\
&= & \log\bigl(\zeta_i'(\r_i)\bigr) + \log(\r-z) + \log\bigl(1+\o(1)\bigr) \\
& = & \log\bigl(\zeta_i'(\r_i)\bigr) + \log(\r-z) + \o(1).
\end{eqnarray*}
We remark that $(1+z)^p$ and $\log(1+z)$ are analytic in a neighbourhood of $z=0$.
In the following we denote by $i\in \{1,2\}$ the index such that $S(z)=S_i(z)$. Then, the computations above imply with Lemma \ref{lemma:F-expansion-rest} that
$$
\sum_{s_i\in\mathrm{supp}(\mu_i)} \mu(s_i)\,z\, E^{(i)}\bigl(s_i| \zeta_i(z)\bigr)=\O_c\bigl((\r-z)^{q_i} \log^{k_i} (\r-z)\bigr).
$$
Since the matrix $M$ from the proof of Proposition \ref{prop:xi-expansion}
is invertible, we can conclude analogously that we must have 
$$
X_{D}^{(1)}(z)= \O_c\bigl((\r-z)^{q_i} \log^{k_i} (\r-z)\bigr) \ \textrm{ and } \  
X_{D}^{(2)}(z) =
\O_c\bigl((\r-z)^{q_i} \log^{k_i} (\r-z)\bigr).
$$
Thus, the leading
singular term  of $\zeta_i(z)$ has the same order as the leading
singular term in the expansion of $G_i(z)$ if $S(z)=S_i(z)$. By (\ref{equ:G-Gi-link}),
we can conclude that the leading singular term in the expansion of $G(z)$ at $z=\r$
 has the same form as the leading singular term  in the expansion of $G_i(z)$
 at $z=\r_i$, namely $(\r-z)^{q_i} \log^{k_i} (\r-z)$.
\par
Recall that we assumed throughout this section that $G_i(z)$ is exactly $d_i$-times differentiable at $z=\r_i$.
For an application of \textit{Darboux's method} we need in a first step the expansion of $G(z)$ in a neighbourhood of $z=\r$ up to terms of order $(\r-z)^{D+2}$, where $D=d_1$, if $\bar\theta<\theta_2/\alpha_2$, and $D=\min\{d_1,d_2\}$, if $\bar\theta=\theta_1/\alpha_1=\theta_2/\alpha_2$. Thus, by (\ref{equ:G-Gi-link}), we have to extend the expansions
of $\zeta_1(z)$ and $\zeta_2(z)$ up to terms of order $(\r-z)^{D+2}$. The next
lemma ensures that there are only finitely many terms up to
order $(\r-z)^{D+2}$ in these expansions.
\begin{Lemma}\label{lemma:expansion-higher-order} 
For $i\in\{1,2\}$, $\zeta_i(z)$ has an expansion of the form
$$
\sum_{k=0}^{D} x_k(\r-z)^k + \sum_{(q,k)\in\mathcal{T}} x_{(q,k)} (\r-z)^q
\log^k(\r-z) + \o\bigl((\r-z)^{D+2}\bigr),
$$
where $x_k,x_{(q,k)}\in\mathbb{R}$, $\mathcal{T}$ is a finite subset of
$\widehat{\mathcal{T}}:=\bigl\lbrace (q,k)\in\mathbb{R}\times\N_0 \mid D<q\leq
D+2\bigr\rbrace$. In particular, $(q_i,k_i)\in\mathcal{T}$ with
$x_{(q_i,k_i)}\neq 0$, and $(q,k)\in\mathcal{T}$ implies $(q_i,k_i)\preceq (q,k)$.
\end{Lemma}
\begin{proof}
Recall the expansion of $\sum_{s_i\in\mathrm{supp}(\mu_i)} \mu_i(s_i)\,z\,
E^{(i)}(s_i|z)$ from Lemma \ref{lemma:F-expansion-rest}.
Assume that $\zeta_i(z)$ has already an expansion of the form
\begin{equation}\label{equ:xi-expansion-lemma}
\sum_{k=0}^{D} x_k(\r-z)^k + \sum_{(q,k)\in\mathcal{T}'} x_{(q,k)} (\r-z)^q
\log^k(\r-z) + \o(\max \mathcal{T}'),
\end{equation}
where $\T'$ is a finite subset of $\widehat\T$ and $\max
\mathcal{T}':=\max_{\preceq}\bigl\lbrace (\r-z)^q\log^k(\r-z) \mid
(q,k)\in\mathcal{T}'\bigr\rbrace$. In particular, $x_{(q_i,k_i)}\in\T'$ with $x_{(q_i,k_i)}\neq 0$. We proceed with expanding the next terms of $\zeta_i(z)$
analogously to the proof of Proposition \ref{prop:xi-expansion}. For this purpose, observe that for $p>1$  we can rewrite
$\bigl(\zeta_i(\r)-\zeta_i(z)\bigr)^p$ as
\begin{equation}\label{equ:root-expansion}
(-x_1)^p\,(\r-z)^p\,\ \biggl(1 +\sum_{k=2}^D \frac{x_k}{x_1} (\r-z)^{k-1} +
\sum_{(q,k)\in\mathcal{T}'} \frac{x_{(q,k)}}{x_1} (\r-z)^{q-1} \log^k(\r-z) +
\o\Bigl(\frac{\max\mathcal{T}'}{\r-z}\Bigr)\biggr)^p
\end{equation}
and $\log\bigl(\zeta_i(\r)-\zeta_i(z)\bigr)$ as
\begin{equation}\label{equ:log-expansion}
C + \log(\r-z) + \log \biggl(1+\sum_{k=2}^{D} \frac{x_k}{x_1}(\r-z)^{k-1} + \sum_{(q,k)\in\mathcal{T}'}
\frac{x_{(q,k)}}{x_1} (\r-z)^{q-1} \log^k(\r-z) + \o\Bigl(\frac{\max\mathcal{T}'}{\r-z}\Bigr)\biggr).
\end{equation}
Note that $(1+z)^p$ with $p>1$ and $\log(1+z)$ are analytic in a neighbourhood
of $z=0$. We substitute (\ref{equ:xi-expansion-lemma}),
(\ref{equ:root-expansion}) and (\ref{equ:log-expansion}) in Equations (\ref{equ:xi1}) and
(\ref{equ:xi2}) and compare again the error terms (we will repeat this
procedure in each of the following steps).
Therefore, if $\max\mathcal{T}'=(\r-z)^{\hat  q} \log^{\hat k}(\r-z)$ then the next possible terms up to order $(\r-z)^{\hat q}$ in the expansion may only be
$$
(\r-z)^{\hat q} \log^{\hat k-1}(\r-z), (\r-z)^{\hat q} \log^{\hat k-2}(\r-z),
\dots, (\r-z)^{\hat q}.
$$
Analogously to the proof of Proposition \ref{prop:xi-expansion} we determine
step by step the corresponding coefficients of these terms. The next term in
the expansion of $\zeta_i(z)$ has now the form $(\r-z)^{\check q} \log^{\check
  k}(\r-z)$, where $\check q>\hat q$ is a sum of elements from the finite set 
$$
\bigl\lbrace 1,q,q-1 \mid (q,\cdot)\in\mathcal{T}_1\cup \mathcal{T}_2 \bigr\rbrace
$$
with $\mathcal{T}_i$ given as in (\ref{equ:expansion-assumption}).
The value of $\check q$ is minimal such that $\check q >\hat q$. Due to (\ref{equ:root-expansion}) and (\ref{equ:log-expansion})
there is obviously a maximal $\check k\in\N_0$ such that $(\r-z)^{\check q}
\log^{\check k}(\r-z)$ may be a non-vanishing next term in the expansion of
$\zeta_i(z)$. Thus, we may iterate the last few steps again. Since there are only
finitely many possible values for $q$ such that a term of the form $(\r-z)^{q}
\log^{k}(\r-z)$ may appear in the expansion up to order $(\r-z)^{D+2}$, we have shown that there are only
finitely many terms up to order $(\r-z)^{D+2}$ in the expansion of $\zeta_i(z)$.
\end{proof}
With the last lemma we are now able to prove Theorem \ref{thm:Psi<0}:
\begin{proof}[Proof of Theorem \ref{thm:Psi<0}].
We start by expanding $\zeta_1(z)$ and $\zeta_2(z)$ as in Proposition
\ref{prop:xi-expansion}. If $\alpha_1 > \theta_1/(\theta_1+\theta_2)$ then
$\bar \theta = \theta_1/\alpha_1<\theta_2/\alpha_2$ and $\zeta_1(\r)=\r_1$, $\zeta_2(\r)<\r_2$, and
consequently the leading singular term in the expansion of $\zeta_1(z)$ (and
$\zeta_2(z)$) is then given by the term $S_1(z)=(\r-z)^{q_1} \log^{k_1}(\r-z)$. Analogously, if we have $\bar\theta =
\theta_2/\alpha_2<\theta_1/\alpha_1$, then $\zeta_2(\r)=\r_2$ and
$\zeta_1(\r)<\r_1$, and the leading singular term is then $S_2(z)=(\r-z)^{q_2}
\log^{k_2}(\r-z)$. If $\alpha_1=\theta_1/(\theta_1+\theta_2)$ then $\bar \theta = \theta_1/\alpha_1=\theta_2/\alpha_2$,
$\zeta_1(\r)=\r_1$, $\zeta_2(\r)=\r_2$, and the leading singular term in the
expansions of $\zeta_1(z)$ and $\zeta_2(z)$ is $S_j(z)=(\r-z)^{q_j} \log^{k_j}(\r-z)$,
where $j=1$, if $S_1(z)\preceq S_2(z)$, and $j=2$, if $S_2(z)\prec S_1(z)$.
 For the rest of the proof, we denote by
$i\in\{1,2\}$ the index such that $S(z)=S_i(z)$. Therefore, the expansion of
the common leading singular term of $\zeta_1(z)$ and $\zeta_2(z)$, namely
$S_i(z)$, in a neighbourhood of $0$ has coefficients of asymptotic order
proportional to $\r^{-n} n^{-\lambda_i}\log^{\kappa_i} n$.
\par

We will use the technique which is called \textit{Darboux's method}: recall that the \textit{Riemann-Lebesgue-Lemma} states that
if a function $H(z)=\sum_{n\geq 0} h_nz^n$ has radius of convergence $\r_H$ and
if $H$ is $k$-times continuously differentiable on its circle of convergence, then $h_n\r_H^n n^k\to 0$ as $n\to\infty$. Thus, one identifies all singularities
on the circle of convergence and subtracts parts of the expansion near them
such that the remaining part is sufficiently often differentiable on the circle.
The asymptotics of the coefficients arise then from the main leading singular
terms. We refer to Olver \cite[Chap. 8, \S  9.2]{olver} for more details.
\par
Lemma \ref{lemma:expansion-higher-order} assures that we have a singular expansion of $\zeta_1(z)$ up to terms of order \mbox{$\lceil \lambda_i\rceil=\lceil q_i\rceil +1=D+2$,} which allows us to apply Darboux's method: we get the asymptotic behaviour of $\mu^{(n\delta)}(e)$ by
plugging $\zeta_1(z)$ into Equation (\ref{equ:G-Gi-link}).
Thus, the leading singular term in the expansion of $G(z)$ in a neighbourhood
of $z=\r$ is the same as the one of
$\zeta_1(z)$, namely $(\r -z)^{q_i}\log^{k_i} (\r -z)$. We have to show that the expansion  of $G(z)$ at every singular point on the disc of convergence has the same form. The singularities are exactly the points $\r \exp(\mathrm{i}2\pi j/\delta)$ with $0\leq j < \delta-1$; see e.g. \cite[Theorem 9.4]{woess}. Writing $z=\lambda \r \omega_j$, where $\omega_j=\exp(\mathrm{i}2\pi  j/\delta)$ and $\lambda\in\mathbb{C}$ with $|\lambda|<1$, 
$$
G(z) = G(\lambda \r \omega_j) = \sum_{n\geq 0} \mu^{(n\delta)}(e) (\lambda \r
\omega_j)^{n\delta} = \sum_{n\geq 0} \mu^{(n\delta)}(e) (\lambda
\r)^{n\delta} = G(\lambda\r)=G(z/\omega_j).
$$
Thus, for every $j\in\{0,1,\dots,\delta-1\}$, we have expansions of $G(z)$ in a neighbourhood of $z=\r\omega_j$ given by
$$
G(z)= \sum_{k=0}^{D} g_{k} (\r -z/\omega_j)^k 
+ \sum_{(q,k)\in\widehat{\mathcal{T}}_i} g_{(q,k)} (\r -z/\omega_j)^{q}\log^{k}(\r-z/\omega_j) + \O\bigl((\r \omega_j-z)^{D+2} \bigr),
$$
where  $\widehat{\mathcal{T}}_i$ is a finite subset of
$\{(q,k)\in\mathbb{R}\times\mathbb{N} \mid D<q\leq D+2,q>q_i \lor (q=q_i
\Rightarrow k<k_i)\}$, $g_{(q_i,k_i)}\in\widehat{\mathcal{T}}_i$ with
$g_{(q_i,k_i)}\neq 0$ and $(q,k)\in\widehat{\mathcal{T}}_i$ implies
$(q_i,k_i)\preceq (q,k)$. Therefore, the difference
$$
G(z)- \sum_{j=0}^{\delta-1}  \sum_{(q,k)\in\widehat{\mathcal{T}}_i} g_{(q,k)} (\r -z/\omega_j)^{q}\log^{k}(\r-z/\omega_j)
$$
is $(D+2)$-times differentiable on the circle of convergence. Observe now that the coefficients of the expansion of $(\r -z/\omega_j)^{q_i}\log^{k_i}(\r -z/\omega_j)$ in a neighbourhood of $0$ behave asymptotically like $C\,(\r\omega_j)^{-n}\,n^{-\lambda_i}\,\log^{\kappa_i}(n)$. We can drop higher order terms in the above difference because the corresponding coefficients have higher asymptotic order. Since $G(z)=\sum_{n\geq 0} \mu^{(n)}z^n $, we can conclude that
$$
\mu^{(n)}(e) \sim \sum_{j=0}^{\delta-1} C\, n^{-\lambda_i}\,\log^{\kappa_i}(n)\, \r^{-n}\, \omega_j^{-n}. 
$$
Observe that $\sum_{j=0}^{\delta-1} \omega_j^{-n}=\delta$ if $\delta$ divides $n$, and this sum is zero otherwise.

We note once again that the asymptotic behaviour of the coefficients in
the expansion of the function \mbox{$(\r-z)^{q_i} \log^{k_i}(\r-z)$} near $0$ are well-known; see e.g. Flajolet and \mbox{Sedgewick \cite{flajolet07}.} 
\end{proof}
Let us remark that the reasoning in the above proof shows analogously the
asymptotic behaviour $\mu_i^{(n)}(e_i) \sim \hat C_i\, \r_i^{-n}\, n^{-\lambda_i}\,\log^{\kappa_i} n$. That is, in the presented case of $\Psi(\bar\theta)>0$, $G_1'(\r_1)<\infty$ and $G_2'(\r_2)<\infty$ the asymptotics are directly inherited from the asymptotics of the random walk on $\Gamma_i$ governed by $\mu_i$.

\section{The Case $\Psi(\bar\theta)=0$}

\label{sec:psi=0}

We now consider the case $\Gamma=\Gamma_1\ast\Gamma_2$ and assume that $\Psi(\bar\theta)=0$, $G_1'(\r_1)<\infty$ and \mbox{$G_2'\bigl(\zeta_2(\r)\bigr)<\infty$} hold. W.l.o.g. we may also assume
$\theta =\bar\theta = \theta_1/\alpha_1$. The aim of this section is to prove the following:
\begin{Th}\label{th:psi-bar-theta=0}
Assume that $G_1'(\r_1)<\infty$ and  $G'_2\bigl(\zeta_2(\r_2)\bigl)<\infty$. If $\Psi(\bar\theta)=0$ then
$$
\mu^{(n\delta)}(e) \sim C\cdot  \r^{-n\delta}\cdot n^{-3/2}.
$$
\end{Th}
In the following we will derive expansions of $\zeta_i(z)$ and $G(z)$ in a neighbourhood of $z=\r$ in order to prove Theorem \ref{th:psi-bar-theta=0}.
Recall from (\ref{equ:Phi-Psi}) that $\Psi(\bar\theta)=0$ implies 
$$
\Phi'(\bar \theta)=\frac{\Phi(\bar \theta)}{\bar\theta}= \frac{\Phi(\theta)}{\theta} =\frac{\Phi\bigl(\r\,G(\r)\bigr)}{\r\,G(\r)} =\frac{G(\r)}{\r\,G(\r)}=\frac{1}{\r}.
$$
Differentiating (\ref{phi_12}) yields
\begin{equation}\label{g_prime}
 G'(z)=\frac{G(z)\Phi'\bigl(zG(z)\bigr)}{1-z\Phi'\bigl(zG(z)\bigr)}.
\end{equation}
Therefore, $G'(\r)=\infty$, and consequently we have to proceed differently
from the previous section in order to find
the expansion of $G(z)$.
First, we show positivity of $\Phi''(\bar\theta)$ in the present setting:
\begin{Lemma}\label{lemma:phi''<infty}
Assume that $G_1'(\r_1)<\infty$ and $G'_2\bigl(\zeta_2(\r_2)\bigr)$.
If $\Psi(\bar\theta)=0$ then $\Phi''(\bar\theta)>0$.
\end{Lemma}
\begin{proof}
Differentiating (\ref{equ:Phi-Psi-formula}) twice yields
\begin{equation}\label{equ:phi''}
\Phi''(\bar\theta)=\alpha_1^2 \Phi_1''(\alpha_1\bar\theta) + \alpha_2^2 \Phi_2''(\alpha_2\bar\theta).
\end{equation}
Since $\Phi_1(t)$ and $\Phi_2(t)$ are strictly convex for $t\in[0,\theta_1)$ and $t\in[0,\theta_2)$ respectively, we get $\Phi''(\bar\theta)>0$ whenever $\theta_1/\alpha_1\neq \theta_2/\alpha_2$: if $\bar\theta=\theta_1/\alpha_1< \theta_2/\alpha_2$ then $\alpha_2\bar\theta<\theta_2$, that is, $\Phi_2''(\alpha_2\bar\theta)>0$. \\
We now consider the case $\theta_1/\alpha_1= \theta_2/\alpha_2$, that is, $\zeta_2(\r)=\r_2 $. Assume now $\Phi''(\bar\theta)=0$. Then $\Phi_1''(\theta_1)=\lim_{t\to\theta_1-}\Phi_1''(t)=0$ and 
$\Phi_2''(\theta_2)=\lim_{t\to\theta_2-}\Phi_2''(t)=0$ must hold. For $i\in\{1,2\}$, differentiating (\ref{phi_12}) yields
$$
G_i'(\r_i)=\lim_{z\to\r_i}\frac{G_i(z) \Phi_i'\bigl(zG_i(z)\bigr)}{1-z\Phi_i'\bigl(zG_i(z)\bigr)},
$$
or equivalently
$$
\Phi_i'(\theta_i)=\lim_{z\to\r_i} \frac{G_i'(z)}{z G_i'(z)+ G_i(z)}=\frac{G_i'(\r_i)}{\r_i G_i'(\r_i)+ G_i(\r_i)}<\infty.
$$
In particular, we have $\Phi_i'(\theta_i)<1/\r_i$ since $G_i'(\r_i)<\infty$ by assumption.
If $\Phi_i''(\theta_i)=0$, differentiating (\ref{phi_12}) twice yields
$$
G_i''(\r_i)=\lim_{z\to\r_i}\frac{\Phi_i''\bigl(zG_i(z)\bigr) \bigl(G_i(z) +zG_i'(z)\bigr)^2 + 2\Phi_i'\bigl(zG_i(z)\bigr) G_i(z)}{1-z\Phi_i'\bigl(zG_i(z)\bigr)}
= \frac{2\Phi_i'(\theta_i) G_i(\r_i)}{1-\r_i \Phi_i'(\theta_i)}<\infty.
$$
Define the first return generating function as 
\[
U_i(z):=\sum_{n\geq 1} \Prob\bigl[X_n^{(i)}=e_i,\forall m\in\{1,\dots,n\}: X_m^{(i)}\neq e_i\mid X_0^{(i)}=e_i\bigr]\,z^n,
\]
which satisfies the well-known equation $G_i(z)=1/\bigl(1-U_i(z)\bigr)$ and is strictly convex. $G_i''(\r_i)<\infty$ implies obviously $U_i''(\r_i)<\infty$.
Therefore, we can compute $\Phi_i''(\theta_i)$ as
$$
\Phi_i''(\theta_i)=\lim_{z\to\r_i} \frac{G_i(z)^3 U_i''(z)}{\bigl(G_i(z)+zG_i'(z)\bigr)^3}=\frac{G_i(\r_i)^3 U_i''(\r_i)}{\bigl(G_i(\r_i)+\r_iG_i'(\r_i)\bigr)^3}>0,
$$
a contradiction, and consequently $\Phi''(\bar\theta)>0$ due to (\ref{equ:phi''}).
\end{proof}
We proceed with expanding $G(z)$ nearby $z=\r$.
\begin{Prop}\label{prop:Psi=0}
Assume that $\Phi''(\bar{\theta}) <\infty $, $\Psi(\bar\theta)=0$, $G_1'(\r_1
)<\infty $ and $G_2'(\zeta_2(\r )) <\infty $ hold. Then we can expand $G(z)$ in a neighbourhood of $z=\r$ as follows:
$$
 G(z) =  g_0 + g_1 \sqrt{\r-z} + \mathbf{o}\bigl(\sqrt{\r-z}\bigr),
$$
where $g_0,g_1\in\mathbb{R}$ with $g_1\neq 0$.
\end{Prop}
\begin{proof}
Consider the auxiliary function $H(z):=\bigl( G(z)-G(\r)\bigr)^2$,
and its first derivative
$
 H'(z)=2G'(z)\bigl(G(z)-G(\r)\bigr)
$.
Using Equation (\ref{g_prime}), we get
\[
 H'(z)=2\frac{G(z)\Phi'\bigl(zG(z)\bigr)}{1-z\Phi'\bigl(zG(z)\bigr)} \bigl(G(z)-G(\r)\bigr).
\]
The next aim is to show differentiability of $H(z)$ at $z=\r$. For this purpose,
we want to show finiteness of the following limit:
\[
 \lim_{z\to \r}H'(z)=\lim_{z\to \r} 2G(z)\Phi'\bigl(zG(z)\bigr)\frac{G(z)-G(\r)}{1-z\Phi'\bigl(zG(z)\bigr)}.
\]
Since $2G(z)\Phi'\bigl(zG(z)\bigr)$ tends to
$A:=2G(\r)/\r<\infty$, we just look at the following 
limit:
\begin{eqnarray}\label{limit}
 \lim_{z\to \r}  \frac{G(z)-G(\r)}{1-z\Phi'\bigl(zG(z)\bigr)} & = & \lim_{z\to
   \r} \frac{\Phi\bigl(zG(z)\bigr)-G(\r)}{1-z\Phi'\bigl(zG(z)\bigr)}\nonumber\\
& = &\lim_{z\to \r} \, \frac{\Phi'\bigl(zG(z)\bigr)\bigl(G(z)+zG'(z)\bigr)}{-\Phi'\bigl(zG(z)\bigr)-z\Phi''\bigl(zG(z)\bigr)\bigl(G(z)+zG'(z)\bigr)}.
\end{eqnarray}
In the last equation we applied De L'H\^opital's rule.
We now write $\mathcal{G}(z):=G(z)+zG'(z)$, which tends to infinity for $z\to\r$. Recall that $\bar{\theta}=\theta =\r G(\r )$ if $\Psi (\bar{\theta})=0 $.
Therefore, Equation (\ref{limit})  
yields
$$
H'(\r)\!=\! \lim_{z\to \r} \frac{A \Phi'(\theta)
   \mathcal{G}(z)}{-\Phi'(\theta)-\r\Phi''(\theta) \mathcal{G}(z)}\! =\!
\lim_{x\to \infty} \frac{A \Phi'(\theta)x}{-\Phi'(\theta)-\r\Phi''(\theta) x}=\frac{A}{-\r^2\Phi''(\theta)} \in (-\infty,0).
$$
Thus,
\[
\lim_{z\to \r} \frac{G(\r)-G(z)}{\sqrt{\r-z}} = \lim_{z\to \r} \sqrt{\frac{\bigl(G(z)-G(\r)\bigr)^2}{\r-z}}=\sqrt{-H'(\r)}\in (0,\infty)
\]
leads to the proposed expansion, namely
$$
G(z) = G(\r)-\sqrt{-H'(\r)}\sqrt{\r-z}+\mathbf{o}(\sqrt{\r-z}),
$$
where $\sqrt{-H'(\r)}\neq 0$. 
\end{proof}
The next lemma shows that also $\zeta_1(z)$ and $\zeta_2(z)$ have the same expansion type:
\begin{Lemma}\label{lemma:xi-i-psi=0}
Assume $\Phi''(\bar{\theta})<\infty $. If  $\Psi(\bar\theta)=0$, $G_1'(\r_1)<\infty $ and $G_2'(\zeta_2(\r)) <\infty $ we can expand $\zeta_1(z)$ and $\zeta_2(z)$ in a neighbourhood of $z=\r$ in the following way:
$$
\zeta_1(z)=\r_1+ a_0 \sqrt{\r-z} +\o(\sqrt{\r-z}) ,\quad \quad \zeta_2(z)=\zeta_2(\r)+b_0 \sqrt{\r-z}+\o(\sqrt{\r-z}),
$$
where $a_0,b_0\in\mathbb{R}\setminus\{0\}$.
\end{Lemma}
\begin{proof}
Obviously, we can write
\begin{equation}\label{equ:xi-expans1}
\zeta_1(z)=\r_1+X_1(z),\quad \quad \zeta_2(z)=\zeta_2(\r)+X_2(z),
\end{equation}
where $X_1(\r)=X_2(\r)=0$.
Moreover, for $i\in\{1,2\}$,
\begin{equation}\label{equ:Gi-expans1}
 G_i\bigl(\zeta_i(z)\bigr)=G_i\bigl(\zeta_i(\r)\bigr) -
 G_i'\bigl(\zeta_i(\r)\bigr)\bigl(-X_i(z)\bigr)+\o\bigl(X_i(z)\bigr).
\end{equation}
Substituting (\ref{equ:xi-expans1}) and
(\ref{equ:Gi-expans1}) in (\ref{equ:G-Gi-link})
yields the claim when comparing all error terms.
\end{proof}
Now we can show that $\Phi''(\bar{\theta}) <\infty $ holds in the present setting:
\begin{Lemma}
Assume $G_1'(\r_1)< \infty$ and $ G_1(\zeta_2(\r))<\infty $. If $\Psi(\bar{\theta})=0$ then $ \Phi'' (\bar{\theta})<\infty $.
\end{Lemma}
\begin{proof}
Assume now that $\Phi''(\bar\theta)=\infty$. We rewrite $\zeta_1(z) $ and $\zeta_2(z)$ as
\begin{equation}\label{equ:xi-i-psi=0}
\zeta_1(z)=\r_1+X_1(z),\quad \textrm{ and }\quad \zeta_2(z)=\zeta_2(\r)+X_2(z),
\end{equation}
with $X_1(\r)=X_2(\r )=0 $. More precisely, if $\Phi''(\bar\theta)=\infty$, then  the reasoning in Proposition \ref{prop:Psi=0} yields $H'(\r)=0 $, and consequently $X_1(z),X_2(z)= \o(\sqrt{\r-z})$. Furthermore, $X_1(z),X_2(z)\neq \O\bigl((\r-z)\bigr)$, because otherwise $\zeta_1'(\r),\zeta_2'(\r)<\infty $ together with (\ref{equ:G-Gi-link}) would lead to a contradiction with $G'(\r) =\infty $.
For $i\in\{1,2\}$ and $s_i\in\mathrm{supp}(\mu_i)$, we write in the following $F_i(s_i|z)=\sum_{n\geq  1}f_n^{(i)}(s_i)z^n$  with suitable coefficients $f_n^{(i)}(s_i)\in\mathbb{R}$. Our next aim is to find real numbers $C_1^{(i)}$ and $C_2^{(i)}$ such that
\begin{equation}\label{first_step}
 C_1^{(i)} X_1(z)+C_2^{(i)} X_2(z)+\o(\r-z)=\operatorname{LP}_i,
\end{equation}
where $\operatorname{LP}_i$ is a linear polynomial. For this purpose, we
rewrite Equations (\ref{equ:xi1}) and (\ref{equ:xi2}) with the help of
(\ref{equ:xi-i-psi=0}). In the following denote by $j$ the element of $\{1,2\}$ which is
different from $i$. We get:
\begin{equation}\label{equ:psi0-equation}
 \Big (1-\alpha_j (\r\!-\!(\r\!-\!z)) \sum_{s_j\in\mathrm{supp}(\mu_j)}\mu_j(s_j)\sum_{n\geq 1}f_n^{(j)}(s_j)\left (\zeta_j(\r)+X_j(z)\Big )^n\right )\! \bigl(\zeta_i(\r)+X_i(z)\bigr) = \alpha_i z .
\end{equation}
The coefficients $C_1^{(i)}$ and $C_2^{(i)}$ of $X_1(z)$ and $X_2(z)$ respectively, are
\begin{eqnarray*}
 C_1^{(1)} &:= &1-\alpha_2 \r
 \sum_{s_2\in\mathrm{supp}(\mu_2)}\mu_2(s_2)\sum_{n\geq
   1}f_n^{(2)}(s_2)\,\zeta_2(\r)^n\\
&=& 1-\alpha_2 \r \sum_{s_2\in\mathrm{supp}(\mu_2)}\mu_2(s_2)F_2\bigl(s_2|\zeta_2(\r)\bigr),\\
C_2^{(1)} &:= & -\alpha_2 \r_1 \r \sum_{s_2\in\mathrm{supp}(\mu_2)}\mu_2(s_2)\sum_{n\geq 1}f_n^{(2)}(s_2)n\, \zeta_2(\r)^{n-1}\\
&=& -\alpha_2 \r_1 \r \sum_{s_2\in\mathrm{supp}(\mu_2)}\mu_2(s_2) F_2' \bigl(s_2|\zeta_2(\r)\bigr),\\
 C_1^{(2)} &:= &  -\alpha_1 \zeta_2(\r) \r
 \sum_{s_1\in\mathrm{supp}(\mu_1)}\mu_1(s_1)\sum_{n\geq 1}f_n^{(1)}(s_1)n
 \r_1^{n-1}\\
&=& -\alpha_1 \zeta_2(\r) \r \sum_{s_1\in\mathrm{supp}(\mu_1)}\mu_1(s_1) F_1' \bigl(s_1|\r_1\bigr),\\
C_2^{(2)} &:= & 1-\alpha_1 \r \sum_{s_1\in\mathrm{supp}(\mu_1)}\mu_1(s_1)\sum_{n\geq 1}f_n^{(1)}(s_1)\r_1^n \\
& = & 1-\alpha_1 \r \sum_{s_1\in\mathrm{supp}(\mu_1)}\mu_1 (s_1)F_1(s_1|\r_1).
\end{eqnarray*}
For $i=1$, the linear polynomial term on the left hand side of (\ref{equ:psi0-equation}) is
\[
 \r_1\Bigl( 1-\alpha_2 z \sum_{s_2\in\mathrm{supp}(\mu_2)}\mu_2(s_2) F_2\bigl(s_2|\zeta_2(\r)\bigr)\Bigr),
\]
while on the right hand side it is $\alpha_1 z$. For $i=2$, we have on the left
hand side of (\ref{equ:psi0-equation})
$$
\zeta_2(\r) \Bigl( 1-\alpha_1 z \sum_{s_1\in\mathrm{supp}(\mu_1)}\mu_1(s_1) F_1(s_1|\r_1)\Bigr),
$$
and on the right hand side $\alpha_2 z$.
Therefore, (\ref{first_step}) holds with
\begin{eqnarray*}
 \operatorname{LP}_1 & := & \alpha_1 z-\r_1\Bigl( 1-\alpha_2 z \sum_{s_2\in\mathrm{supp}(\mu_2)}\mu_2(s_2)F_2\bigl(s_2|\zeta_2(\r)\bigr)\Bigr)\
 \textrm{ and }\\
 \operatorname{LP}_2 & := & \alpha_2 z-\zeta_2(\r) \Bigl( 1-\alpha_1 z \sum_{s_1\in\mathrm{supp}(\mu_1)}\mu_1(s_1) F_1(s_1|\r_1)\Bigr).
\end{eqnarray*}
The coefficients $C_1^{(i)}, C_2^{(i)}$ satisfy 
\begin{equation}\label{det=0}
C_1^{(1)} C_2^{(2)} - C_1^{(2)} C_2^{(1)} = 0.
\end{equation}
Indeed, assume that $C_1^{(1)} C_2^{(2)} - C_1^{(2)} C_2^{(1)} \neq 0$. Then the following linear system
\begin{eqnarray*}
 C_1^{(1)} X_1(z)+C_2^{(1)} X_2(z)+\textbf{o}(\r-z) & = &\operatorname{LP}_1,\\
 C_1^{(2)} X_1(z)+C_2^{(2)} X_2(z)+\textbf{o}(\r-z) & = &\operatorname{LP}_2
\end{eqnarray*}
has a unique solution for $X_1(z)$ and $X_2(z)$, but this means that both of them are
of order $\O(\r-z)$, a contradiction to (\ref{equ:xi-i-psi=0}), where
$X_1(z),X_2(z)\neq \O(\r-z)$.
\par
Evaluating Equation (\ref{equ:psi0-equation}) with $i=2$ at $z=\r$ gives $C_2^{(2)}>0$. Equation (\ref{det=0}) yields
\begin{equation}\label{lt1-lt2}
 \operatorname{LP}_1-\frac{C_2^{(1)}}{C_2^{(2)}}\operatorname{LP}_2=0.
\end{equation}
Evaluating the last equation at $z=0$ yields
\begin{equation}\label{computed2c2}
-\r_1 + \frac{C_2^{(1)}}{C_2^{(2)}} \cdot \zeta_2(\r)  = 0 .
\end{equation}
Since $C_2^{(1)}<0$, Equation (\ref{computed2c2}) gives us a contradiction, therefore $\Phi''(\bar{\theta})=\infty$ cannot hold when $\Psi(\bar{\theta})=0 $.
\end{proof}

We now proceed analogously to the previous section: we substitute the expansion
of the last lemma in Equations (\ref{equ:xi1}) and (\ref{equ:xi2}) and
determine step by step the next terms in the expansions of $\zeta_1(z)$ and
$\zeta_2(z)$. The next lemma shows that we get only a finite number of terms up
to order $(\r-z)^2$: 
\begin{Lemma}\label{lem:high_ord_II}
Let $i\in\{1,2\}$. If $\Psi(\bar\theta)=0$, we can expand $\zeta_i(z)$ in a neighbourhood of $z=\r$ in the following way:
$$
\zeta_i(z) = \zeta_i(r) + c_0 \sqrt{\r-z} 
%+ c_1 (\r-z) 
+ \sum_{(q,k)\in\mathcal{T}}
c_{(q,k)} (\r-z)^q \log^k(\r-z) + \O\bigl((\r-z)^2\bigr),
$$
where $\mathcal{T}$ is a finite subset of $\widehat{\mathcal{T}} :=\bigl\lbrace
(q,k)\in\mathbb{R}\times \N_0 \mid 1/2<q\leq 2\bigr\rbrace$ and $c_0,c_{(q,k)}\in\mathbb{R}$ with $c_0\neq 0$.
\end{Lemma}
\begin{proof}
We start by plugging $\zeta_i(z)= \zeta_i(\r)+c_0\sqrt{\r-z}+X_0^{(i)}(z)$ with $X_0^{(i)}(z)=\mathbf{o}(\sqrt{r-z})$ into
Equations (\ref{equ:xi1}) and (\ref{equ:xi2}) and determine step by step the next terms inductively
analogously to the proof of Lemma
\ref{lemma:expansion-higher-order}.
Assume now that $\zeta_i(z)$ has an expansion of the form
$$
\zeta_i(\r) + c_0 \sqrt{\r-z} %+ c_1(\r-z)
+ \sum_{(q,k)\in\mathcal{T}'} c_{(q,k)}
(\r-z)^q \log^k(\r-z) + \o(\max \mathcal{T}'),
$$
where $\T'$ with $\T'\subseteq\widehat{\T}$ finite. For $p>1$,
$\bigl(\zeta_i(\r)-\zeta_i(z)\bigr)^p$ can be rewritten as
\begin{equation}\label{equ:root-expansion2} 
(-c_0)^p\,(\r-z)^{p/2}\,\biggl(1 %+ \frac{c_1}{c_0} \sqrt{\r-z} 
+ \sum_{(q,k)\in\mathcal{T}'} \frac{c_{(q,k)}}{c_0} (\r-z)^{q-1/2} \log^k(\r-z) +
\o\Bigl(\frac{\max \mathcal{T}'}{\sqrt{\r-z}}\Bigr)\biggr)^p
\end{equation}
and $\log\bigl(\zeta_i(\r)-\zeta_i(z)\bigr)$ as
\begin{equation}\label{equ:log-expansion2} 
C + \frac{1}{2}\log(\r-z) + \log \biggl(1+%\frac{c_1}{c_0}\sqrt{\r-z} +
\sum_{(q,k)\in\mathcal{T}'} \frac{c_{(q,k)}}{c_0} (\r-z)^{q-1/2} \log^k(\r-z) +\o\Bigl(\frac{\max \mathcal{T}'}{\sqrt{\r-z}}\Bigr) \biggr).
\end{equation} 
Once again, if $\max \mathcal{T}'=(\r-z)^{\hat q} \log^{\hat k}(\r-z)$ then the next possible terms up to order $(\r-z)^{\hat q}$ in the expansion may only be
$$
(\r-z)^{\hat q} \log^{\hat k-1}(\r-z), (\r-z)^{\hat q} \log^{\hat k-2}(\r-z),
\dots, (\r-z)^{\hat q}.
$$
We determine step by step the corresponding coefficients of these terms by
plugging the expansions of $\zeta_i(z)$, (\ref{equ:root-expansion2}) and (\ref{equ:log-expansion2}) into Equations
(\ref{equ:xi1}) and (\ref{equ:xi2}) and comparing error terms. The
next term has the form $(\r-z)^{\check q} \log^{\check k}(\r-z)$,
where $\check q\leq 2$ is now a sum of elements from the finite set
$\bigl\lbrace 1/2,q/2,q/2-1/2 \mid (q,\cdot)\in\mathcal{T}_1 \cup \mathcal{T}_2
\bigr\rbrace$ such that $\check q >\hat q$ (recall the definitions of
$\mathcal{T}_i$ from (\ref{equ:expansion-assumption})). Due to (\ref{equ:root-expansion2}) and (\ref{equ:log-expansion2})
there is obviously a maximal $\check k\in\N_0$ such that $(\r-z)^{\check q}
\log^{\check k}(\r-z)$ may be a non-vanishing next term in the expansion of
$\zeta_i(z)$. Iterating the last steps yields the claim of the lemma, since
there are only finitely many possible values for $q$ such that the term $(\r-z)^{q}
\log^{k}(\r-z)$ may appear in the expansion of $\zeta_i(z)$.
\end{proof}
Substituting the obtained expansion of $\zeta_1(z)$ into Equation (\ref{equ:G-Gi-link})
yields the proposed claim of Theorem \ref{th:psi-bar-theta=0}.
\par
\textit{Remark:} The result could also be obtained analogously to
Flajolet and Sedgewick \cite[Section VI.7.]{flajolet07} by singularity
analysis, but one still has to prove positivity and finiteness of $\Phi''(\bar\theta)$.

\section{The remaining Cases}
\label{sec:lowdim}

In this section we look at all remaining cases not covered by Section
\ref{sec:mainresults} and \ref{sec:psi=0}. Afterwards we will extend our results to free products $\Gamma_1\ast\ldots\ast \Gamma_m $ with $m>2$.

\subsection{Case $G_1(\r_1)<\infty$ and $G_1'(\r_1)=\infty$}
\label{subsec:d1=3,4}
\begin{Th}\label{thm:5.1}
Consider a free product of the form $\Gamma_1 \ast \Gamma_2$, where
$G_1(\r_1)<\infty$, $G'_1(\r_1)=\infty$ and $G_2'(\r_2)<\infty$. Then:
$$
\mu^{(n\delta)}(e) \sim
\begin{cases}
C_1\cdot \r^{-n\delta}\cdot n^{-3/2}, & \textrm{if } \bar \theta =
\theta_1/\alpha_1\ \textrm{ or }\ \Psi(\bar\theta) \leq 0,\\[2ex]
C_2\cdot \r^{-n\delta}\cdot n^{-\lambda_2}\cdot \log^{\kappa_2}(n), & \textrm{if } \bar \theta = \theta_2/\alpha_2<
\theta_1/\alpha_1 \textrm{ and } \Psi(\bar\theta) >0.
\end{cases}
$$
\end{Th}
\begin{proof}
For the first part of the proof assume that
$\bar\theta=\theta_1/\alpha_1$. With 
$$
U_1(z):=\sum_{g\in\Gamma_1}\mu_1(g)\,z\,F_1(g^{-1}|z)
$$ 
we have the well-known equation $G_1(z)=1/\bigl(1-U_1(z)\bigr)$. Therefore, $G_1'(\r_1)=\infty$ implies $U_1'(\r_1)=\infty$, and we get due to 
\cite[Equation (9.14)]{woess} 
\begin{equation}\label{equ:psi1=0}
\Psi_1(\alpha_1 \bar\theta)=
\Psi_1(\theta_1)
= \lim_{z\to\r_1}\Psi_1\bigl(z G(z)\bigr) = \lim_{z\to\r_1} \frac{1}{z U_1'(z) + 1 - U_1(z)} = 0.
\end{equation}
Thus,
$$
\Psi(\bar \theta) 
 = \Psi_1(\alpha_1 \bar\theta) + \Psi_2(\alpha_2 \bar\theta)-1
=\Psi_1(\theta_1) + \Psi_2(\alpha_2 \bar\theta)-1 = \Psi_2(\alpha_2 \bar\theta)-1. 
$$
Recall that $\Psi(t)$ is strictly decreasing and $\Psi_2(0)=1$. Therefore,
$\Psi(\bar\theta)<0$, and consequently we obtain the asymptotic behaviour
$\mu^{(n\delta)}(e) \sim C_1 \r^{-n\delta} n^{-3/2}$; see \cite[Theorem 17.3]{woess}. 
\par
For the case 
$\bar\theta=\theta_2/\alpha_2 <\theta_1/\alpha_1$ and $\Psi(\bar\theta)=0$, we
refer to Section \ref{sec:psi=0}. 
\par
In the case $\bar \theta = \theta_2/\alpha_2< \theta_1/\alpha_1$ and
$\Psi(\bar\theta) >0$ the Green function $G_1(z)$ is analytic at $z=\zeta_1(\r)<\r_1$ and
thus we may apply the technique from Section \ref{sec:mainresults} to obtain
the proposed asymptotic behaviour.
\end{proof}
At this point, let us remark that the formula for $\Psi(t)$ used in Equation (\ref{equ:psi1=0}) always implies $\Psi_i(\theta_i)=0$ whenever $G_i'(\r_i)=\infty$. Moreover:
\begin{Cor}\label{Cor:5.2}
If $G_1'(\r_1)=G_2'(\r_2)=\infty$, then $\mu^{(n\delta)}(e) \sim C\cdot \r^{-n\delta}\cdot
n^{-3/2}$.
\end{Cor}
\begin{proof}
Since $U'_1(\r_1)=U'_2(\r_2)=\infty$, Equation (\ref{equ:psi1=0}) implies that at least one of
$\Psi_1(\alpha_1\bar\theta)$ and $\Psi_2(\alpha_2\bar\theta)$ equals zero,
yielding $\Psi(\bar\theta)<0$.
\end{proof}

\subsection{Case $G_1(\r_1)=\infty$}

For finite groups $\Gamma_1$ and $\Gamma_2$, Woess \cite{woess3} proved
that the $n$-step return probabilities behave asymptotically like
$C\cdot \r^{-n\delta}\cdot n^{-3/2}$. Moreover, we get the following asymptotic behaviours:
\begin{Th}\label{thm:5.3}
Consider a free product of the form $\Gamma_1 \ast \Gamma_2$, where
$G_1(\r_1)=\infty$. Then:
$$
\mu^{(n\delta)}(e) \sim
\begin{cases}
C_1\cdot \r^{-n\delta}\cdot n^{-3/2}, & \textrm{if } 
\Psi(\bar\theta) \leq 0,\\
C_2 \cdot \r^{-n\delta} \cdot n^{-\lambda_2}\cdot \log^{\kappa_2}(n), & \textrm{if } \Psi(\bar\theta)> 0.
\end{cases}
$$
\end{Th} 
\begin{proof}
If $G_2'(\r_2)=\infty$, we have $\Psi(\bar\theta)<0$; see proof of Corollary \ref{Cor:5.2}.
\par
If $G_2(\r_2)<\infty$ and $G_2'(\r_2)=\infty$ then
$\bar\theta=\theta_2/\alpha_2$, and $U'_2(\r_2)=\infty$. This implies once
again $\Psi(\alpha_2\bar\theta)=0$, and thus $\Psi(\bar\theta)<0$.
\par
If $G_2'(\r_2)<\infty$ then $\bar\theta=\theta_2/\alpha_2$ and $\zeta_1(\r)<\r_1$. Therefore, we can follow the
argumentation of \mbox{Section \ref{sec:mainresults}} and \ref{sec:psi=0} analogously to prove the proposed claim.
\end{proof}

\subsection{Free Products with more than two Factors}

\label{sec:higher-order}

Let $m\in\N$ with $m\geq 3$. Suppose we are given finitely generated groups $\Gamma_1,\dots,\Gamma_m$. We consider now a free product of the form
$\Gamma:=\Gamma_1\ast \ldots \ast \Gamma_m$, on which
%Analogously to
%the setting in the case $m=2$, suppose we are given probability
%measures $\mu_i$ on $\Gamma_i$ with $\langle \mathrm{supp}(\mu_i)\rangle =\Gamma_i$. Let $\alpha_1,\dots,\alpha_{m}>0$
%with \mbox{$\sum_{j=1}^{m}\alpha_j=1$.} Then we consider the 
a random walk is governed by the measure $\mu$ defined as $\mu:=\sum_{j=1}^{m}\alpha_j
\bar\mu_j$; see Section \ref{sec:RW}. We get the following result:
\begin{Th}\label{thm:higher-order}
Let $m\geq 3$. Consider the free product $\Gamma:=\Gamma_1\ast \ldots \ast \Gamma_m$ equipped with a random walk governed by
$\mu:=\sum_{j=1}^{m}\alpha_j \bar\mu_j$. Assume that the corresponding Green functions $G_i(z)$ on the free factors $\Gamma_i$ have an expansion as in (\ref{equ:expansion-assumption}) whenever $G_i'(\r)<\infty$. Denote by $\r$ the radius of convergence
of the Green function associated with the random walk on $\Gamma$. \
% \par
Then the asymptotic behaviour of the
corresponding $n$-step transition probabilities must obey one of the following
laws: $C\, \r^{-n\delta}\, n^{-\lambda_i}\,\log^{\kappa_i}(n)$, where $\lambda_i$
and $\kappa_i$ are inherited from one
of the $\mu_i$'s, or $C\, \r^{-n\delta}\, n^{-3/2}$ with some constant $C=C_\mu$ depending on $\mu$.
\end{Th}
\begin{proof}
In order to prove the theorem, we just remark that -- by induction on the number of free
factors -- the Green function (with radius of convergence $\r^\ast$) of the random walk on $\Gamma^\ast
:=\Gamma_1\ast \ldots \ast \Gamma_{m-1}$ governed by
$\mu^\ast:=\sum_{j=1}^{m-1}\frac{\alpha_j}{\alpha_1+\ldots +\alpha_{m-1}} \bar\mu_j$ has an expansion either of the form
\begin{equation}
\tag{I}
 G^\ast(z) = \sum_{k=0}^{D} g_k(\r^\ast-z)^k + \sum_{(q,k)\in\mathcal{T}}
 g_{(q,k)} (\r^\ast-z)^q \log^k(\r^\ast-z) + \O\bigl((\r^\ast -z)^{D+2}\bigr),
\end{equation}
where $\mathcal{T}$ is a finite subset of $\{(q,k)\in\mathbb{R}\times\N_0 \mid
D<q\leq D+2\}$ and $g_k,g_{(q,k)}\in\mathbb{R}$, or of the form
\begin{equation}
\tag{II} G^\ast(z) = g_0 + g_1 \sqrt{\r^\ast-z} + %g_2 (\r^\ast-z) +
\sum_{(q,k)\in\mathcal{T}}  g_{(q,k)} (\r^\ast-z)^q \log^k(\r^\ast-z) + \O\bigl((\r^\ast -z)^{2}\bigr),
\end{equation}
where $\mathcal{T}$ is a finite subset of $\{(q,k)\in\mathbb{R}\times\N_0 \mid
1/2<q\leq 2\}$ and $g_0,g_1,g_{(q,k)}\in\mathbb{R}$. 
Thus, we may apply the results from Section \ref{sec:mainresults}
to the free product $\Gamma^\ast \ast \Gamma_m$ equipped with
$\mu=(\alpha_1+\ldots +\alpha_{m-1})\mu^\ast +\alpha_m \bar\mu_m$ and obtain the proposed result.
\end{proof}

\section{Examples}
\label{sec:exmaples}

\subsection{Free Products of Lattices}
\label{sec:Z-lattices}

Let $d_1,\dots,d_m\in\N$. In this subsection we consider free products of
the form \mbox{$\Gamma:=\Z^{d_1}\ast\ldots \ast \Z^{d_m}$,} equipped with a
nearest neighbour random walk, that is, we always assume
$\mathrm{supp}(\mu_i)=\{\pm e_j^{(i)}\mid 1\leq j \leq d_i \}$, where $e_j^{(i)}$ is the $j$-th unit vector in $\Z^{d_i}$. In the following
subsection we show that the Green functions of nearest neighbour random walks
on $\Z^d$ have an expansion as requested by (\ref{equ:expansion-assumption}). Afterwards we can give a
complete classification of the asymptotic behaviour.

\subsubsection{Expansion of the Green Function on $\Z^d$}

Let $d\in \N$. Suppose we are given a probability measure $\pi$ with
$\mathrm{supp}(\pi)=\{\pm e_1,\dots,\pm e_{d}\}$, the set of
natural generators of $\Z^d$. Then $\pi$ defines a random
walk on $\Z^d$, and we denote by $\pi^{(n)}$ its $n$-fold convolution
power. We write for $1\leq i \leq d$
$$
\beta_i := \pi(e_i) + \pi(-e_i) \quad \textrm{ and }  \quad
 p_i:=\frac{\pi(e_i)}{\pi (e_i) +\pi (-e_i)}.
$$ 
Denote by $\mathbf{0}$ the zero vector in $\Z^d$. Once again $G_d(z):=\sum_{n\geq
  0} \pi^{(n)}(\mathbf{0}) z^n$ denotes the associated Green
function, which has radius of convergence $\r_d$.
The crucial point for our later discussion is the following:
\begin{Prop}\label{prop:expansion-Gd}
The Green function of the random walk on $\Z^d$ has an expansion of the form
$$
G_d(z) =
\begin{cases}
f(z) + g(z) (\r_d-z)^{(d-2)/2}, & \textrm{if } d \textrm{ is odd,}\\
f(z) + g(z) (\r_d-z)^{(d-2)/2}\log (\r_d-z), & \textrm{if } d \textrm{ is even,}
\end{cases}
$$
where the functions $f(z),g(z)$ are analytic in a neighbourhood of $z=\r_d$ and $g(\r_d)\neq 0$. 
\end{Prop}
\textit{Remarks:} For the case of simple random walks on $\Z^d$, i.e. $\pi(\pm e_i)=1/(2d)$, a proof of this proposition
can be found in \cite[Proposition 17.16]{woess}. In our case, we generalize the statement to arbitrary nearest neighbour random walks on $\Z^d$, but we will only give a sketch of the proof and refer once again to \cite{woess}. From the expansion follows with the help of Darboux's method that $\pi^{(2n)}(\mathbf{0})\sim C\,\r_d^{-2n}\,n^{-d/2}$; this asymptotic behaviour follows also from Cartwright and Soardi \cite{cartwright-soardi87}.
\begin{proof}
First, note that the spectral radius of the random walk on $\Z^d$ is given by
$$
\varrho = \sum_{i=1}^d \beta_i \sqrt{4p_i(1-p_i)}=\frac{1}{\r_d};
$$
compare with \cite[Theorem 8.23]{woess}. For $i\in\{1,\dots,d\}$, we define random walks on $\Z$
governed by probability measures $\pi_i$ with $\pi_i(1):=p_i$ and
$\pi_i(-1):=1-p_i$. For $z\in\C$, the \textit{exponential generating function} on $\mathbb{Z}^d$ is given by
\[
 E(z):=\sum_{n=0}^{\infty}\pi^{(n)} (\mathbf{0})\frac{z^n}{n!}
\]
and on the $i$-th coordinate axis it is given by
$$
E_i(z) := \sum_{n\geq 0}\pi_i^{(n)}(0) \frac{z^n}{n!} = \int_{-1}^1
e^{\sqrt{4p_i(1-p_i)} tz} \frac{1}{\pi \sqrt{1-t^2}}dt.
$$
In the last equation we applied the following relation, which is easy to check:
$$
\pi_i^{(n)}(0) = \int_{-1}^1 \sqrt{ 4 p_i (1-p_i)}^{n} t^{n} \frac{1}{\pi \sqrt{1-t^2}} dt.
$$
Furthermore, we get $E(z) = \prod_{i=1}^d E_i(\beta_i z)=\int_{-\varrho}^{\varrho} e^{tz} \bigl(\hat f_1 \ast \ldots \ast \hat
f_d\bigr)(t) dt$,
where 
$$
\hat f_i(t):= \frac{1}{\beta_i\sqrt{4p_i(1-p_i)}} f_0\Bigl(\frac{t}{\beta_i\sqrt{4p_i(1-p_i)}}\Bigr)
\textrm{ and } \
f_0(t):= 
\begin{cases}
\frac{1}{\pi \sqrt{1-t^2}}, & \textrm{ if } t\in(-1,1),\\
0, & \textrm{otherwise}.
\end{cases}
$$
This allows us to rewrite the Green function in the following way:
\begin{equation}\label{equ:green-integral}
G_d(z) = \int_{-\varrho}^\varrho \frac{1}{1-zt} \bigl(\hat f_1 \ast \ldots \ast
\hat f_d\bigr)(t)\,dt .
\end{equation}
Moreover, there is a function $g_d(t)$, which is analytic in a neighbourhood of
$t=\varrho$ and satisfies $g_d(\varrho)\neq 0$ such that
\begin{equation}\label{equ:density-expansion}
\bigl(\hat f_1 \ast\ldots \ast \hat f_d\bigr)(t) = (\varrho -t)^{(d-2)/2} g_d(t).
\end{equation}
To prove this, we define
$\bar f_i(t):= \hat f_i\bigl(\beta_i \sqrt{4 p_i(1-p_i)} -t\bigr)$ 
and show inductively that we can write
$$
(\bar f_1 \ast \ldots \ast \bar f_d)(t)=t^{(d-2)/2} \bar
g_d(t),
$$
where the function $\bar g_d(t)$ is analytic in a neighbourhood of $t=0$ and $\bar
g_d(0)\neq 0$. Analogously to the proof of \cite[Proposition 17.16]{woess}, we
may conclude together with (\ref{equ:green-integral}) and
(\ref{equ:density-expansion}) that $G_d(z)$ has the proposed expansion.
\end{proof}

\subsubsection{Classification of the Asymptotic Behaviour}

Observe that a nearest neighbour random walk on $\Z^d$ has period $2$ since it can return to the origin only in an even number of steps. Therefore, the period of a nearest neighbour random walk on $\Z^{d_1}\ast \Z^{d_2}$ is $\delta=2$.
Now we can give a complete classification of the asymptotic behaviour of
$n$-step return probabilities of nearest neighbour random walks on
$\Z^{d_1}\ast \Z^{d_2}$:
\begin{Th}
Consider irreducible nearest neighbour random walks on the lattices $\Z^{d_1}$ and
$\Z^{d_2}$ with $d_1\leq d_2$. Then the $n$-step return probabilities of the associated random
walk on $\Z^{d_1}\ast \Z^{d_2}$ obey one the following laws:
$$
\mu^{(2n)}(e) \sim
\begin{cases}
C_1 \cdot  \r^{-2n}\cdot n^{-d_1/2}, & \textrm{if } d_1\geq 5 \textrm{ and } \Psi(\bar\theta)>0 \textrm{ and } \bar\theta=\theta_1/\alpha_1,\\
C_2 \cdot \r^{-2n}\cdot n^{-d_2/2}, & \textrm{if } d_2\geq 5 \textrm{ and } \Psi(\bar\theta)>0 \textrm{ and } \bar\theta=\theta_2/\alpha_2<\theta_1/\alpha_1,\\
C_3 \cdot \r^{-2n}\cdot n^{-3/2}, & \textrm{otherwise}.
\end{cases}
$$
\end{Th}
\begin{flushright}$\Box$\end{flushright}

Consider now the multi-factor free product $\Z^{d_1}\ast \ldots \ast \Z^{d_m}$.
Let $\mu_i$ be the simple random walk on $\Z^{d_i}$ for each
$i\in\{1,\dots,m\}$, that is, $\mu_i(\pm e_j^{(i)})=1/(2d_i)$, where $e_j^{(i)}$ is the $j$-th unit vector in $\Z^{d_i}$. Choose $\alpha_1,\dots,\alpha_m>0$ with $\sum_{j=1}^m\alpha_j=1$. Let $G_i(z)$ denote the Green function of the simple
random walk on $\Z^{d_i}$, which has radius of convergence $\r_i=1$, and define $\Psi_i(t)$ analogously as in (\ref{equ:Phi-Psi}). 
Cartwright \cite{cartwright88} computed numerically some of the values of $\Psi_i\bigl(G_i(1)\bigr)$ and
showed that $\Psi_i\bigl(G_i(1)\bigr)\to 1$ if $d_i\to\infty$. Thus, for large
$d_i$ we have $\Psi_i\bigl(G_i(1)\bigr)>1-1/m$. Recall also that
$\Psi_i(t)$ is decreasing. Denote by $G(z)$ the Green
function of the random walk on  $\Z^{d_1}\ast \ldots \ast \Z^{d_m}$ and by $\r$ its
radius of convergence, and define $\Psi(t)$ analogously as in
(\ref{equ:Phi-Psi}). By \cite[Equation 9.21]{woess}, 
$$
\Psi(\bar \theta) = 1+ \sum_{j=1}^m \bigl(\Psi_i(\alpha_i \bar\theta) -1\bigr),
$$
where $\bar\theta=\min_{1\leq i\leq m} \theta_i/\alpha_i$. If all exponents
$d_i\geq 5$ are large enough, we get \mbox{$\Psi(\bar \theta)>0$.} Furthermore, if $\alpha_i$
is chosen large enough, we get an asymptotic behaviour of the form
$C_i\,\r^{-2n}\,n^{-d_i/2}$. Moreover, one can define (symmetric) measures
$\mu_1,\dots,\mu_m$ supported on the natural generators in such a way that we
obtain a $C_0\,\r^{-2n}\,n^{-3/2}$-law: one chooses $\mu_1$ and $\mu_2$ such
that $\Psi_1(\theta_1),\Psi_2(\theta_2)<1/2$, and $\alpha_1$ and
$\alpha_2$ are chosen such that $\bar\theta=\theta_1/\alpha_1=\theta_2/\alpha_2$, yielding 
$$
\Psi(\bar \theta)=
1 + \underbrace{\bigl(\Psi_1(\theta_1)-1\bigr)}_{< -1/2} + \underbrace{ \bigl(\Psi_2(\theta_2)-1\bigr)}_{<-1/2}
+ \underbrace{\sum_{k=3}^m \bigl(\Psi_k(\alpha_k\bar \theta)-1\bigr)}_{\leq 0}
<0;
$$
see also comments at the end of \mbox{Section \ref{sec:RW}.} That is, we can have $m+1$ different asymptotic
behaviours. This finally proves Theorem \ref{thm:1.1}.
\par
For instance, consider $\Gamma=\Z^5\ast \Z^6\ast \Z^7$ equipped with simple random
walks $\mu_1$, $\mu_2$ and $\mu_3$ on each free factor. For $i\in\{1,2,3\}$, we define $\Psi_i(t)$ analogously to (\ref{equ:Phi-Psi}). \mbox{Cartwright \cite{cartwright88}} computed the values
$\Psi_1\bigl(G_1(1)\bigr)=0.691$, $\Psi_2\bigl(G_2(1)\bigr)=0.824$ and $\Psi_{3}\bigl(G_3(1)\bigr)=0.876$. Thus, the random walk on $\Z^5\ast \Z^6$ governed by $\mu_{12}:=\alpha_1^\ast \bar\mu_1 +\alpha_2^\ast\bar\mu_2$, where $\alpha_1^\ast=\alpha_1/(\alpha_1+\alpha_2)$ and $\alpha_2^\ast=\alpha_2/(\alpha_1+\alpha_2)$, satisfies $\Psi(M)\geq 0.515$ with $M:=\min\{\theta_1/\alpha_1^\ast,\theta_2/\alpha_2^\ast\}$. That is, $M=\r_{1,2}G_{1,2}(\r_{1,2})$, where $G_{1,2}(z)$ is the Green function of the random walk on $\Z^5\ast \Z^6$ with radius of convergence $\r_{1,2}$. 
Since all $\Psi_i$-functions are strictly decreasing, we obtain for the random walk on $\Gamma=\Gamma_1\ast\Gamma_2$ with $\Gamma_1=\Z^5\ast\Z^6$ and $\Gamma_2=\Z^7$:
\begin{eqnarray*}
\Psi(\bar\theta) & = & \Psi_{1}\bigl((\alpha_1+\alpha_2) \bar\theta\bigr)+\Psi_{2}(\alpha_3 \bar\theta)-1
\geq 0.515 + 0.876-1 >0.
\end{eqnarray*}
For the simple random walk on $\Gamma$, we have then the asymptotic non-exponential
type $n^{-7/2}$, if $\alpha_1+\alpha_2< M/\bigl(M+G_3(1)\bigr)$. Otherwise, we
have the asymptotic behaviour $n^{-5/2}$, if $M=\theta_1/\alpha_1^\ast$, or
$n^{-3}$, if $M=\theta_2/\alpha_2^\ast\neq \theta_1/\alpha_1^\ast$.

\subsection{$(\Z/m\Z) \ast \Z^d$}

Consider the groups $\Gamma_1=\Z/m\Z$ and $\Gamma_2=\Z^d$ for any
$m,d\in\N$ with $m\geq 2$. Suppose we are given a probability measure $\mu_1$ on $\Gamma_1$
and a probability measure $\mu_2$ on $\Z^d$, which is supported on the natural
generators. Then $G_1(1)=\infty$, and thus we get the following
classification:
$$
\mu^{(n\delta)}(e) \sim
\begin{cases}
C_1\cdot \r^{-n\delta}\cdot  n^{-d/2}, & \textrm{if } \Psi(\bar\theta)>0,\\
C_2\cdot  \r^{-n\delta}\cdot  n^{-3/2}, & \textrm{otherwise}.
\end{cases}
$$
Let us remark that $\Psi(\bar\theta)<0$ if $d\leq 4$: this follows from the
fact $G_2'(\r_2)=\infty$ (see Proposition \ref{prop:expansion-Gd}) and Corollary \ref{Cor:5.2}.

\subsection{$\Pi_q \ast \Z^d$}

Consider the groups $\Gamma_1=\Pi_q:=\ast_{i=1}^q (\Z/2\Z)$ and $\Gamma_2=\Z^d$
for any $q,d\in\N$ with $q\geq 2$. Observe that the Cayley graph of $\Gamma_1$ is the homogeneous tree
of degree $q$.
Suppose we are given probability measures $\mu_1$ on $\Gamma_1$
and $\mu_2$ on $\Z^d$, which are both supported on the natural
generators. If $q=2$ then $G_1(1)=\infty$, and thus we get the same
classification as in the case $(\Z/m\Z) \ast \Z^d$. If $q\geq 3$, then it is
well-known that $G_1(z)$ can be written as
$$
G_1(z) = A(z) +\sqrt{\r_1-z}\, B(z),
$$
where $A(z),B(z)$ are analytic in a neighbourhood of $z=\r_1$ and $B(\r_1)\neq
0$; see e.g. Woess \cite[Equation (4.5)]{woess03}. Therefore, we get the following classification for the associated random
walk on the free product $\Gamma_1\ast\Gamma_2$:
$$
\mu^{(2n)}(e) \sim
\begin{cases}
C_1\cdot  \r^{-2n}\cdot  n^{-d/2}, & \textrm{if } 
\bar\theta=\theta_2/\alpha_2< \theta_1/\alpha_1 \textrm{ and }\Psi(\bar\theta)>0,\\
C_2\cdot  \r^{-2n}\cdot  n^{-3/2}, & \textrm{otherwise}.
\end{cases}
$$
Analogously to the previous example, observe that $d\leq 4$ implies $\Psi(\bar\theta)<0$.

\section{Classification of Phase Transitions}
\label{sec:phase-transition}

Let us return to the case $m=2$, that is, $\Gamma= \Gamma_1\ast \Gamma_2$. We now \emph{fix} the measures $\mu_1$ and $\mu_2$, and investigate the variation of $\Psi(\bar{\theta})$ as a function of the parameter $\alpha_1$.
\begin{Lemma}\label{lemma:Upsilon}
Assume $\bar \theta <\infty$. Then the function $\Upsilon: (0,1)\mapsto \mathbb{R}$ defined by 
$$
\Upsilon(\alpha_1):=\Psi_1(\alpha_1\bar\theta)+\Psi_2\bigl((1-\alpha_1)\bar\theta\bigr)-1
$$ 
is continuous, strictly decreasing in the interval $\bigl (0,\frac{\theta_1}{\theta_1+\theta_2}\bigr ]$ and strictly increasing in the interval $\bigl [\frac{\theta_1}{\theta_1+\theta_2},1\bigr )$. (We set $\frac{c}{c+\infty}:=0$ and $\frac{\infty}{\infty+c}:=1$ for $c\in(0,\infty)$.)
\end{Lemma}
\begin{proof}
We leave the proof of  continuity of $\Upsilon$ as an easy exercise to the
reader, since $\Psi_i$ is analytic in an open neighbourhood of the interval
$[0,\theta_i)$.
\par 
Note that $\Upsilon(\alpha_1)$ equals $\Psi(\bar\theta)$ in dependence of $\alpha_1$. We divide the proof into two parts, according to finiteness of $\theta_1$ and $\theta_2$.\smallbreak
\emph{Case $\theta_1,\theta_2<\infty$}. \ If $0<\alpha_1<\frac{\theta_1}{\theta_1+\theta_2}$ then $\bar\theta = \theta_2/\alpha_2$. Consequently, we have 
$$
\Upsilon(\alpha_1)=\Psi_1\Bigl(\frac{\alpha_1}{1-\alpha_1}\theta_2\Bigr)+\Psi_2(\theta_2)-1.
$$
Since the function $\frac{\alpha_1}{1-\alpha_1}$ is strictly increasing, it follows that $\Psi_1(\frac{\alpha_1}{1-\alpha_1}\theta_2)$ is strictly decreasing, implying $\Upsilon(\alpha_1)$ strictly decreasing.\\
If $\alpha_1=\frac{\theta_1}{\theta_1+\theta_2}$ we obtain $\bar\theta=\theta_1/\alpha_1=\theta_2/\alpha_2$, that is, $\Upsilon(\alpha_1) = \Psi_1(\theta_1)+\Psi_2(\theta_2)-1$.\\
If $\frac{\theta_1}{\theta_1+\theta_2}<\alpha_1<1$ we have $\Psi(\bar\theta)=\Psi_1(\theta_1)+\Psi_2(\frac{1-\alpha_1}{\alpha_1}\theta_1)-1$. Since $\frac{1-\alpha_1}{\alpha_1}$ is strictly decreasing, $\Upsilon(\alpha_1)$ is a strictly increasing function in the abovementioned interval.
\smallbreak
\emph{Case $\theta_1=\infty$}. Then $\bar\theta=\frac{\theta_2}{1-\alpha_1}$. The same reasoning as before shows that $\Upsilon(\alpha_1)$ is strictly decreasing in the interval $(0,1)$.
\smallbreak
\emph{Case $\theta_2=\infty$}. Then $\bar\theta=\frac{\theta_1}{\alpha_1}$. Analogously, $\Upsilon(\alpha_1)$ is strictly increasing in the interval $(0,1)$.
\end{proof}
Let us remark that $\bar\theta=\infty$ implies $\Psi(\bar\theta)<0$ (see \cite[Theorem 9.22]{woess}); otherwise we would have a contradiction to $\rho$-transience.
\smallbreak
Now we can give a complete picture of the \emph{phase transition} of the asymptotic behaviour of the return probabilities depending on the parameter $\alpha_1$, and we present specific examples. 
In the following we discuss the different possible behaviours of the function $\Upsilon(\alpha_1)=\Psi(\bar \theta)$. In Figure \ref{fig:Upsilon}, the dashed line will represent \emph{approximately} the qualitative behaviour of $\Upsilon(\alpha_1)$; we denote its zeros (if they exist) by $\alpha_{\textnormal{low}}$ and $\alpha_{\textnormal{high}}$ (with $\alpha_{\textnormal{low}} \leq \alpha_{\textnormal{high}}$). Moreover, we write $\alpha_{\textnormal{c}}:=\theta_1/(\theta_1+\theta_2)$. We decompose the interval $(0,1)$ into subintervals such that every choice of $\alpha_1$ in a fixed subinterval leads to the same non-exponential type. With the help of Figure \ref{fig:Upsilon} we discuss case by case the different behaviours of $\Upsilon(\alpha_1)$, and for each case we give an example of a nearest neighbour random walk on $\Z^{d_1}\ast \Z^{d_2} $. Recall that $\Psi(0)=\Psi_i(0)=1$.
\begin{figure}[!h]
\begin{minipage}[t]{0.45\linewidth}
\vspace{4ex}
\psset{xunit=0.68cm,yunit=0.68cm,runit=0.68cm}
\begin{flushright}
% % % parabola (1)
\begin{pspicture}*(-4,-3)
\begin{pspicture}(9.5,0)
\linethickness{0.5mm}
\parabola[linestyle=dashed](-0.7,1)(1.5,-2.5)
\end{pspicture}
\end{pspicture}
\end{flushright}
\setlength{\unitlength}{0.5cm}
\begin{flushright}
% % % interval of intersection (1)
\begin{picture}(8,-1.9)
  \linethickness{0.5mm}
\put(-3, 3){\line(1, 0){10}}
\put(-3, 1.5){\line(0, 1){3.5}}
\put(-3.5,5.2){\small $\Upsilon(\alpha_1)$}
\put(-3.17, 4.51){$\mathbf{\uparrow}$}
\put(7.3,3.3){$\alpha_1$}
\put(-3.5,2.1){$0$}
\put(7.1,2.1){$1$}
\put(-3, 2.75){\line(0, 1){0.5}}
\put(2, 2.75){\line(0, 1){0.5}}
\put(2, 2.2){$\alpha_{\textnormal c}$}
\put(0.4, 2.75){\line(0, 1){0.5}}
\put(3.5, 2.75){\line(0, 1){0.5}}
\put(7, 2.75){\line(0, 1){0.5}}
  \end{picture}
\end{flushright}
\begin{flushright}
% % % interval projected (1)
\begin{picture}(8,2.5)
  \linethickness{0.5mm}
\put(-3, 3){\line(1, 0){10}}
\put(7.3,3.3){$\alpha_1$}
\put(-3.3,2.1){$0$}
\put(7.1,2.1){$1$}
\put(0.1, 2.1){$\alpha_{\textnormal{low}}$}
\put(3.3, 2.1){$\alpha_{\textnormal{high}}$}
% \color{blue}
\put(-3, 2.75){\line(0, 1){0.5}}
\put(0.4, 2.75){\line(0, 1){0.5}}
\put(3.5, 2.75){\line(0, 1){0.5}}
\put(0.35, 2.75){\line(1, 0){0.19}}
\put(0.35, 3.25){\line(1, 0){0.19}}
\put(3.55, 2.75){\line(-1, 0){0.19}}
\put(3.55, 3.25){\line(-1, 0){0.19}}
\put(7, 2.75){\line(0, 1){0.5}}
% \color{red}
\put(-2.95,3.5){\footnotesize{$n^{\!-\lambda_2}\! \log^{\kappa_2}\! n$}}
\put(1.5,3.5){\footnotesize{$n^{-\frac{3}{2}}$}}
\put(3.75,3.5){\footnotesize{$n^{\!- \lambda_1}\! \log^{\kappa_1}\! n$}}
  \end{picture}
\end{flushright}
\vspace{-6ex}
%\caption{Type A.}
\begin{center}Type A\end{center}
\end{minipage} \hfill
\begin{minipage}[t]{0.45\linewidth}
\vspace{3ex}
\psset{unit=.35pt}
% % % dashed line (3)
\begin{pspicture}*(335, 160)
  \psset{linewidth=1pt}
  \pscircle[linestyle=dashed](415, 200){220} 
\end{pspicture}
\begin{flushright}
\setlength{\unitlength}{0.5cm}
% % % interval of intersection (3)
\begin{picture}(8,0)
  \linethickness{0.5mm}
\put(-3, 2){\line(1, 0){10}}
\put(-3, 1){\line(0, 1){4.5}}
\put(-3.5,5.8){\small $\Upsilon(\alpha_1)$}
\put(-3.17, 5.01){$\mathbf{\uparrow}$}
\put(7.4,2.3){$\alpha_1$}
\put(-3.5,1.1){$0$}
\put(7.1,1.1){$1$}
\put(-3, 1.75){\line(0, 1){0.5}}
\put(1.17, 1.75){\line(0, 1){0.5}}
\put(7, 1.75){\line(0, 1){0.5}}
  \end{picture}
\end{flushright}
\begin{flushright}
\setlength{\unitlength}{0.5cm}
% % % interval projected (3)
\begin{picture}(8,2.2)
  \linethickness{0.5mm}
\put(-3, 2){\line(1, 0){10}}
\put(7.4,2.3){$\alpha_1$}
\put(-3.3,1.1){$0$}
\put(7.1,1.1){$1$}
\put(1.17, 1.1){$\alpha_{\textnormal{low}}$}
% \color{blue}
\put(-3, 1.75){\line(0, 1){0.5}}
\put(1.17, 1.75){\line(0, 1){0.5}}
\put(7, 1.75){\line(0, 1){0.5}}
\put(1.12, 1.75){\line(1, 0){0.19}}
\put(1.12, 2.25){\line(1, 0){0.19}}
% \color{red}
\put(3,2.3){\footnotesize{$n^{-\frac{3}{2}}$}}
\put(-2.5,2.3){\footnotesize{$n^{\!- \lambda_2}\! \log^{\kappa_2}\! n$}}
  \end{picture}
\end{flushright}
\vspace{-4ex}
\begin{center}Type B\end{center}
\end{minipage}

\vspace*{1cm}
\begin{minipage}[t]{0.45\linewidth}
\vspace{5ex}
\psset{unit=.35pt}
% % % dashed line (2)
\begin{pspicture}*(450, 160)
  \psset{linewidth=1pt}
  \pscircle[linestyle=dashed](157.5, 170){234.5} 
\end{pspicture}
\begin{flushright}
\setlength{\unitlength}{0.5cm}
% % % interval of intersection (2)
\begin{picture}(8,0)
  \linethickness{0.5mm}
\put(-3, 2){\line(1, 0){10}}
\put(-3, 1){\line(0, 1){3.5}}
\put(-3.5,4.7){\small $\Upsilon(\alpha_1)$}
\put(-3.17, 4.01){$\mathbf{\uparrow}$}
\put(7.4,2.3){$\alpha_1$}
\put(-3.5,1.1){$0$}
\put(7.1,1.1){$1$}
\put(-3, 1.75){\line(0, 1){0.5}}
\put(3.15, 1.75){\line(0, 1){0.5}}
\put(7, 1.75){\line(0, 1){0.5}}
  \end{picture}
\end{flushright}
\begin{flushright}
\setlength{\unitlength}{0.5cm}
% % % interval projected (2)
\begin{picture}(8,2.2)
  \linethickness{0.5mm}
\put(-3, 2){\line(1, 0){10}}
\put(7.4,2.3){$\alpha_1$}
\put(-3.3,1.1){$0$}
\put(7.1,1.1){$1$}
\put(3.1, 1.1){$\alpha_{\textnormal{high}}$}
% \color{blue}
\put(-3, 1.75){\line(0, 1){0.5}}
\put(7, 1.75){\line(0, 1){0.5}}
\put(3.21, 1.75){\line(-1, 0){0.19}}
\put(3.21, 2.25){\line(-1, 0){0.19}}
\put(3.15, 1.75){\line(0, 1){0.5}}
% \color{red}
\put(-1,2.3){\footnotesize{$n^{-\frac{3}{2}}$}}
\put(3.65,2.3){\footnotesize{$n^{\!- \lambda_1}\! \log^{\kappa_1}\! n$}}
  \end{picture}
\end{flushright}
\vspace{-3ex}
%\caption{Type B.}
\begin{center}Type C\end{center}
\end{minipage}\hfill
\begin{minipage}[t]{0.45\linewidth}
\vspace{6ex}
\psset{xunit=0.68cm,yunit=0.68cm,runit=0.68cm}
\begin{flushright}
% % % parabola (4)
\begin{pspicture}*(-1,-1.4)
\begin{pspicture}(7.9,0)
\linethickness{0.5mm}
\parabola[linestyle=dashed](1,1.3)(3.5,-1.3)
\end{pspicture}
\end{pspicture}
\end{flushright}
\setlength{\unitlength}{0.5cm}
\begin{flushright}
% % % interval of intersection (4)
\begin{picture}(8,1.8)
  \linethickness{0.5mm}
\put(-3, 1){\line(1, 0){10}}
\put(-3, 0){\line(0, 1){3.4}}
\put(-3.5,3.8){\small $\Upsilon(\alpha_1)$}
\put(-3.17, 2.91){$\mathbf{\uparrow}$}
\put(7.3,1.25){$\alpha_1$}
\put(2, 0.75){\line(0, 1){0.5}}
\put(-3.5,0.1){$0$}
\put(7.1,0.1){$1$}
\put(-3, 0.75){\line(0, 1){0.5}}
\put(7, 0.75){\line(0, 1){0.5}}
  \end{picture}
\end{flushright}
\begin{flushright}
% % % interval projected (4)
\begin{picture}(8,2)
  \linethickness{0.5mm}
\put(-3, 1){\line(1, 0){10}}
\put(7.3,1.25){$\alpha_1$}
\put(-3.3,0.1){$0$}
\put(7.1,0.1){$1$}
\put(2, 0.1){$\alpha_{\textnormal{c}}$}
% \color{blue}
\put(-3, 0.75){\line(0, 1){0.5}}
\put(2, 0.75){\line(0, 1){0.5}}
\put(7, 0.75){\line(0, 1){0.5}}
\put(1.95, 0.75){\line(1, 0){0.19}}
\put(1.95, 1.25){\line(1, 0){0.19}}
% \color{red}
\put(-2.15,1.3){\footnotesize{$n^{\!- \lambda_2}\! \log^{\kappa_2}\! n$}}
\put(3.25,1.3){\footnotesize{$n^{\!- \lambda_1}\! \log^{\kappa_1}\! n$}}
  \end{picture}
\end{flushright}
\vspace{-1ex}
\begin{center}Type D\end{center}
\end{minipage}

\vspace*{1cm}
\begin{minipage}[t]{0.45\linewidth}
\vspace*{8ex}
\psset{xunit=0.68cm,yunit=0.68cm,runit=0.68cm}
\begin{flushright}
% % % parabola (5)
\begin{pspicture}(9.5,0)
\linethickness{0.5mm}
\parabola[linestyle=dashed](1.6,-1.15)(5.1,-2.15)
%\parabola[linestyle=dashed](2.15,-1.15)(5.2,-2.15)
\end{pspicture}
\end{flushright}
\setlength{\unitlength}{0.5cm}
\begin{flushright}
% % % interval of intersection (5)
\begin{picture}(8,1.9)
  \linethickness{0.5mm}
\put(-3, 1){\line(1, 0){10}}
\put(-3, -0.7){\line(0, 1){3.2}}
\put(-3.5,3){\small $\Upsilon(\alpha_1)$}
\put(-3.17, 2.01){$\mathbf{\uparrow}$}
\put(7.3,1.25){$\alpha_1$}
\put(-3.5,0.1){$0$}
\put(7.1,0.1){$1$}
\put(-3, 0.75){\line(0, 1){0.5}}
\put(7, 0.75){\line(0, 1){0.5}}
  \end{picture}
\end{flushright}
\begin{flushright}
% % interval projected (5)
\begin{picture}(8,2.8)
  \linethickness{0.5mm}
\put(-3, 1){\line(1, 0){10}}
\put(7.3,1.25){$\alpha_1$}
\put(-3.3,0.1){$0$}
\put(7.1,0.1){$1$}
% \color{blue}
\put(-3, 0.75){\line(0, 1){0.5}}
\put(7, 0.75){\line(0, 1){0.5}}
% \color{red}
\put(1.5,1.3){\footnotesize{$n^{-\frac{3}{2}}$}}
  \end{picture}
\end{flushright}
\vspace{-2ex}
\begin{center}Type E\end{center}
%\vspace*{0.5cm}
\end{minipage} \hfill
\begin{minipage}[t]{0.45\linewidth}
\vspace*{6ex}
\psset{xunit=0.68cm,yunit=0.68cm,runit=0.68cm}
\begin{flushright}
% % % parabola (6)
\begin{pspicture}*(-4,-3)
\begin{pspicture}(9.5,0)
\linethickness{0.5mm}
\parabola[linestyle=dashed](-0.7,1)(1.5,-1.53)
\end{pspicture}
\end{pspicture}
\end{flushright}
\setlength{\unitlength}{0.5cm}
\begin{flushright}
% % % interval of intersection (6)
\begin{picture}(8,-1.9)
  \linethickness{0.5mm}
\put(-3, 3){\line(1, 0){10}}
\put(-3, 1.5){\line(0, 1){3.5}}
\put(-3.5,5.2){\small $\Upsilon(\alpha_1)$}
\put(-3.17, 4.51){$\mathbf{\uparrow}$}
\put(7.3,3.3){$\alpha_1$}
\put(-3.5,2.1){$0$}
\put(7.1,2.1){$1$}
\put(-3, 2.75){\line(0, 1){0.5}}
\put(2, 2.75){\line(0, 1){0.5}}
\put(2, 2.25){$\alpha_{\textnormal c}$}
\put(7, 2.75){\line(0, 1){0.5}}
  \end{picture}
\end{flushright}
\begin{flushright}
% % % interval projected (6)
\begin{picture}(8,2.5)
  \linethickness{0.5mm}
\put(-3, 3){\line(1, 0){10}}
\put(7.3,3.3){$\alpha_1$}
\put(-3.3,2.1){$0$}
\put(7.1,2.1){$1$}
\put(2, 2.25){$\alpha_{\textnormal c}$}
% \color{blue}
\put(-3, 2.75){\line(0, 1){0.5}}
\put(2, 2.75){\line(0, 1){0.5}}
\put(7, 2.75){\line(0, 1){0.5}}
% \color{red}
\put(-2.65,3.5){\footnotesize{$n^{\!-\lambda_2}\! \log^{\kappa_2}\! n$}}
\put(1.5,4){\footnotesize{$n^{-\frac{3}{2}}$}}
\put(1.83,3.5){\footnotesize{$\mathbf{\downarrow}$}}
\put(3.45,3.5){\footnotesize{$n^{\!- \lambda_1}\! \log^{\kappa_1}\! n$}}
  \end{picture}
\end{flushright}
\vspace{-6ex}
\begin{center}Case F\end{center}
\end{minipage}
\caption{The different behaviours of $\Upsilon: \alpha_1\mapsto \Psi(\bar\theta)$.}
\label{fig:Upsilon}
\end{figure}
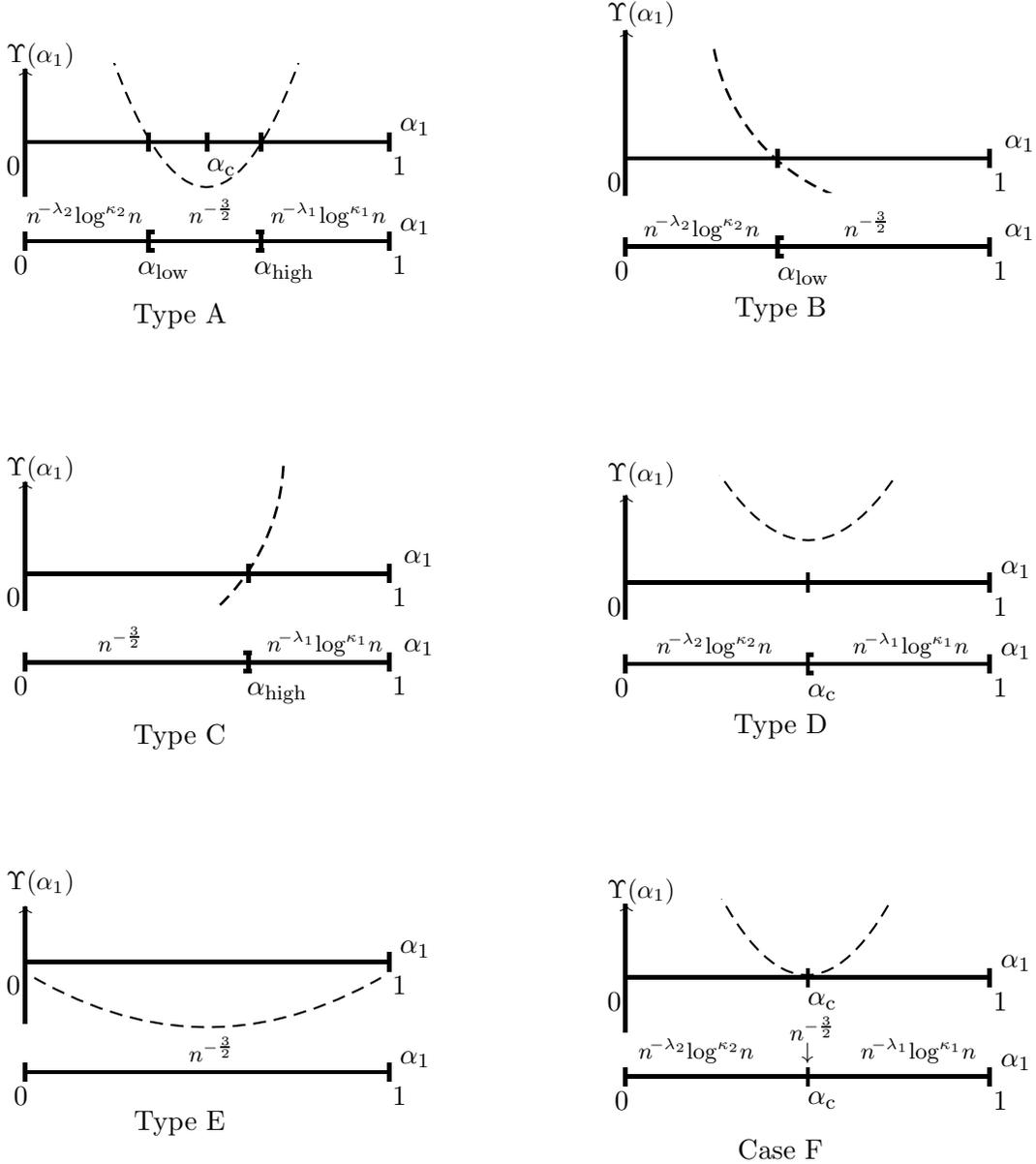

\begin{itemize}
\item[Case A:] Consider Figure \ref{fig:Upsilon}, Case A. We give an example such that this case holds. We set $\Gamma=\Z^{d_1}\ast\Z^{d_2}$ with $d_1,d_2\geq 5$, and we choose $\mu_1$ and $\mu_2$ such that $\Psi_1(\theta_1)<1/2$ and $\Psi_2(\theta_2)<1/2$. Recall that it is  possible to find such measures (see end of Section \ref{sec:RW} and \cite[Lemma 17.9]{woess}). We remark that $\Psi_i(\theta_i)>0$: indeed, $\Psi_i(\theta_i)=0$ would imply 
$$
\Phi_i'(\theta_i)=\frac{\Phi_i(\theta_i)}{\theta_i}=\frac{G_i(\r_i)}{\r_iG_i(\r_i)}=\frac{1}{\r_i}.
$$
Differentiating (\ref{phi_12}) would yield $G_i'(\r_i)=\infty$, a contradiction
to Proposition \ref{prop:expansion-Gd}, according to which, $G_i'(\r_i)$ must
be finite due to $d_i\geq 5 $. Therefore:
\begin{itemize}
 \item If $\alpha_1$ is small then $\bar\theta=\theta_2/(1-\alpha_1)$ and
\begin{equation}\label{equ:typeA}
\Psi(\bar\theta)=\underbrace{\Psi_1\Bigl(\underbrace{\alpha_1 \frac{\theta_2}{1-\alpha_1}}_{\xrightarrow{\alpha_1\to 0} 0}\Bigr)}_{\xrightarrow{\alpha_1\to 0} 1}+ \underbrace{\Psi_2(\theta_2)}_{>0}-1,
\end{equation}
that is, $\Psi(\bar\theta)>0$ if $\alpha_1$ is sufficiently small. This yields a $n^{-d_2/2}$-law for small values of $\alpha_1$. 
\item If $\alpha_1$ is close to $1$ then $\bar\theta=\theta_1/\alpha_1$ and we get analogously an $n^{-d_1/2}$-law.
\item For $\alpha_1=\alpha_c$, we get $\Psi(\bar\theta)=\Psi_1(\theta_1)+\Psi_2(\theta_2)-1<0$, that is, we have a $n^{-3/2}$-law in this case.
\end{itemize}
\item[Case B:] We set $\Gamma=\Z^{2}\ast\Z^{7}$. By Lemma \ref{lemma:Upsilon}, $\Upsilon(\alpha_1)$ is strictly decreasing and $\bar\theta=\theta_2/\alpha_2$.
\begin{itemize}
\item If $\alpha_1$ is small then the same reasoning as in (\ref{equ:typeA}) holds and  $\Psi(\bar\theta)>0$, that is, we have a $n^{-d_2/2}$-law for small $\alpha_1$.
\item If $\alpha_1$ is close to $1$ then
$$
\Psi(\bar\theta)=\underbrace{\Psi_1\Bigl(\underbrace{\alpha_1 \frac{\theta_2}{1-\alpha_1}}_{\xrightarrow{\alpha_1\to 1} \infty}\Bigr)}_{\xrightarrow{\alpha_1\to 1} 0}+ \underbrace{\Psi_2(\theta_2)}_{<1}-1<0,
$$
since $\lim_{t\to\infty}\Psi_1(t)=0$, which follows analogously to (\ref{equ:psi1=0}). That is, we have a $n^{-3/2}$-law for large $\alpha_1$.
\end{itemize}
\item[Case C:] By setting $\Gamma=\Z^{7}\ast\Z^{2}$, we have the symmetric situation as in Case B, which gives an example for this case by exchanging the roles of $\Z^2$ and $\Z^7$.
\item[Case D:] We set $\Gamma=\Z^{5}\ast\Z^{6}$ and consider simple random walks on the factors $\Z^5 $ and $\Z^6 $. By Cartwright \cite{cartwright88}, we have $\Psi_1(\theta_1)=0.691$ and $\Psi_2(\theta_2)=0.824$. Since $\Psi_1(z)$ and $\Psi_2(z)$ are strictly decreasing, we have $\Upsilon(\alpha_1)\geq \Psi_1(\theta_1)+\Psi_2(\theta_2)-1>0$ for all $\alpha_1\in(0,1)$. Thus, we obtain a $n^{-5/2}$-law, if $\alpha_1 \geq\alpha_c$, and a $n^{-3}$-law, if $\alpha_1< \alpha_c$.
\item[Case E:] We set $\Gamma=\Z^{3}\ast\Z^{4}$. By Equation (\ref{equ:psi1=0}), follows that $\Psi_1(\alpha_1\bar\theta)=0$ or $\Psi_2(\alpha_2\bar\theta)=0$, that is, we have $\Upsilon(\alpha_1)<0$ for all $\alpha_1\in (0,1)$. This yields a $n^{-3/2}$-law for all $\alpha_1\in (0,1)$.
\end{itemize}
We now give an example (see Case F of Figure \ref{fig:Upsilon}) where the $n^{-3/2}$-interval of case A collapses to a
singleton. For this purpose, we have to prove the following:
\begin{Lemma}
Consider $\Gamma=\Z^5\ast \Z^6$. Then there are probability measures $\mu_1 $ and $\mu_2 $ supported on the natural generators of $\Z^5 $ and $\Z^6 $ respectively such that \mbox{$\Psi_1(\theta_1)= \Psi_2(\theta_2)=\frac{1}{2}$.}
\end{Lemma}
\begin{proof}
Let $i\in \{1,2 \} $. We have $d_1=5, d_2=6 $ and choose any $\delta\in(0,1)$. We define
$$
\nu_\delta^{(i)}(x) :=
\begin{cases}
(1-\delta)/2, & \textrm{if } x=(\pm1,0,\dots,0)\in\Z^{d_i},\\
\frac{\delta}{2d_i-2}, & \textrm{if }
x=(0,\dots,0,\pm1,0,\dots,0)\in\Z^{d_i} \setminus\{(\pm 1,0,\dots,0)\}.
\end{cases}
$$
The Green function associated with the random walk on $\Z^{d_i}$ governed by
the symmetric measure $\nu_\delta^{(i)}$ has radius of convergence $\r_i=1$; see
\cite[Cor. 8.15]{woess}. If
$\delta=1-1/{d_i}$ then $\Psi_1(\theta_1)=0.691>1/2$ and
$\Psi_2(\theta_2)=0.824>1/2$; see Cartwright \cite{cartwright88}. On
the other hand side, if $\delta$ is small enough then $\Psi_1(\theta_1)<1/2$ and
$\Psi_2(\theta_2)<1/2$; see proof of \cite[Lemma 17.9]{woess}. It remains
to show that $\Psi_i(\theta_i)$ varies continuously in dependence of $\delta$,
which implies that there is some $\delta_0^{(i)}$ such that $\Psi_i(\theta_i)=1/2$. We now write $G_i(z)=G_i(\delta|z)$, $U_i(z)=U_i(\delta|z)$ and
$\Psi_i(t)=\Psi_i(\delta|t)$. Recall that
$$
\Psi_i(\delta|\theta_i)=\frac{1}{U_i'(\delta|1)+1-U_i(\delta|1)}.
$$
Since $U_i(\delta|1)$ can be rewritten as a power series in the variable
$\delta$, the function $\delta \mapsto \Psi_i(\delta|\theta_i)$ is continuous in
$\delta$. This finishes the proof.
\end{proof}
We can now present an example, where Case F of Figure \ref{fig:Upsilon} holds:
we set $\Gamma=\Z^5\ast \Z^6$ and choose the measures $\mu_1$ and $\mu_2$ such
that $\Psi_1(\theta_1)=\Psi_2(\theta_2)=1/2$. Obviously, we have then
$\Upsilon(\alpha_c)=\Psi_1(\theta_1)+\Psi_2(\theta_2)-1=0$. That is, we get the
following asymptotic behaviour:
$$
\mu^{(2n)}(e) \sim 
\begin{cases}
C_1 \cdot \r^{-2n} \cdot n^{-5/2}, & \textrm{if } \alpha_1>\alpha_c,\\
C_2 \cdot \r^{-2n} \cdot n^{-3/2}, & \textrm{if } \alpha_1=\alpha_c,\\
C_3\cdot \r^{-2n} \cdot n^{-3}, & \textrm{if } \alpha_1<\alpha_c.
\end{cases}
$$
As a final remark let us explain that it is not possible that
$\Upsilon(\alpha_1)$ is strictly increasing or decreasing with
$\Upsilon(\alpha_1)>0$ for all $\alpha_1\in(0,1)$. In order to show this assume that $\Upsilon(\alpha_1)$ is strictly increasing. Then, by Lemma \ref{lemma:Upsilon}, $\theta_2=\infty$ must hold, that is, $G_2(\r_2)=\infty$. The same reasoning as in Equation (\ref{equ:psi1=0}) leads to $\lim_{z\to \r_2}\Psi_2\bigl(zG(z)\bigr)=\lim_{t\to \infty}\Psi_2 (t)=0$. Therefore, we obtain for $\alpha_1$ small enough
$$
\Psi(\bar\theta)=\underbrace{\Psi_1(\theta_1)}_{<1}+
\underbrace{\Psi_2\Bigl(\underbrace{(1-\alpha_1) \frac{\theta_1}{\alpha_1}}_{\xrightarrow{\alpha_1\to 0} \infty}\Bigr)}_{\xrightarrow{\alpha_1\to 0} 0}-1<0.
$$
Analogously, if $\Upsilon (\alpha_1) $ is strictly decreasing, then it must have a zero.

\section{Higher Asymptotic Orders}
\label{sec:remarks}

The techniques we used for determining the asymptotic behaviour give us not only
the leading term $n^{-\lambda}\log^{\kappa} n$, but also the proceeding terms of
\textit{higher order}, according to the singular terms in the expansion following the leading one. For
instance, consider a nearest neighbour random walk on $\Z^{7}\ast\Z^8$ with
$\alpha_1=\theta_1/(\theta_1+\theta_2)$. Then the associated Green function has
the following expansion:
\begin{eqnarray*}
&&\sum_{k=0}^4 g_k (\r-z)^4+ \hat g_1 (\r-z)^{5/2} + \check g_1
(\r-z)^3\log(\r-z)  \\
&&\quad + \hat g_2 (\r-z)^{7/2} + \check g_2
(\r-z)^4\log(\r-z) +\o\bigl((\r-z)^4\bigr),
\end{eqnarray*}
where $\hat g_1 \neq 0$. That is,
$$
\mu^{(2n)}(e) \sim  \r^{-2n}\cdot \bigl(C_1\, n^{-7/2} + C_2\, n^{-4} + C_3\,
n^{-9/2} + C_4\, n^{-5} + \o(n^{-5})\bigr),
$$
where $C_1\neq 0$.

\section*{Acknowledgements}

The authors are grateful to the Research and Technology Office at TU Graz,
NAWI-Graz, the German Research Foundation (DFG) grant GI 746/1-1 and the
Austrian Academy of Science (\"OAW), as well as to the European Science
Foundation (ESF) for the activity ``Random Geometry of Large Interacting Systems and Statistical Physics'', for supporting this project.

\bibliographystyle{abbrv}
\bibliography{literatur}

\begin{thebibliography}{10}

\bibitem{cartwright88}
D.~Cartwright.
\newblock Some examples of random walks on free products of discrete groups.
\newblock {\em Ann. Mat. Pura Appl. (4)}, 151:1--15, 1988.

\bibitem{cartwright}
D.~Cartwright.
\newblock On the asymptotic behaviour of convolution powers of probabilities on
  discrete groups.
\newblock {\em Monatsh. Math.}, 107:287--290, 1989.

\bibitem{cartwright-soardi}
D.~Cartwright and P.~Soardi.
\newblock Random walks on free products, quotients and amalgams.
\newblock {\em Nagoya Math. J.}, 102:163--180, 1986.

\bibitem{cartwright-soardi87}
D.~Cartwright and P.~Soardi.
\newblock A local limit theorem for random walks on the cartesian product of
  discrete groups.
\newblock {\em Boll. Unione Mat. Ital. A VII}, 1(1):107--115, 1987.

\bibitem{chatterji-pittet-saloffcoste}
I.~Chatterji, C.~Pittet, and L.~Saloff-Coste.
\newblock Connected {L}ie groups and property {R}{D}.
\newblock {\em Duke Math. J.}, 137(3):511--536, 2007.

\bibitem{flajolet86}
P.~Flajolet and A.~Odlyzko.
\newblock Singularity analysis of generating functions.
\newblock {\em SIAM J. Alg. and Discr. Meth.}, 3(2):216--240, 1990.

\bibitem{flajolet07}
P.~Flajolet and R.~Sedgewick.
\newblock {\em Analytic {C}ombinatorics}.
\newblock Cambridge University Press, 2007.

\bibitem{gerl2}
P.~Gerl.
\newblock A local central limit theorem on some groups.
\newblock In {\em The first Pannonian Symposium on Mathematical Statistics},
  Lect. Notes Statistics 8, pages 73--82. Springer, 1981.

\bibitem{gerl-woess}
P.~Gerl and W.~Woess.
\newblock Local limits and harmonic functions for nonisotropic random walks on
  free groups.
\newblock {\em Probab. Theory Related Fields}, 71:341--355, 1986.

\bibitem{gilch:07}
L.~A. Gilch.
\newblock Rate of escape of random walks on free products.
\newblock {\em J. Aust. Math. Soc.}, 83(I):31--54, 2007.

\bibitem{gilch:09}
L.~A. Gilch.
\newblock Asymptotic entropy of random walks on free products.
\newblock {\em Electron. J. Probab.}, 16:76--105, 2011.

\bibitem{lalley93}
S.~Lalley.
\newblock Finite range random walk on free groups and homogeneous trees.
\newblock {\em Ann. Probab.}, 21(4):2087--2130, 1993.

\bibitem{lalley2}
S.~Lalley.
\newblock Random walks on infinite free products and infinite algebraic systems
  of generating functions.
\newblock Preprint, 2002.

\bibitem{lyndon-schupp}
R.~Lyndon and P.~Schupp.
\newblock {\em Combinatorial Group Theory}.
\newblock Springer-Verlag, 1977.

\bibitem{mairesse1}
J.~Mairesse and F.~Math\'eus.
\newblock Random walks on free products of cyclic groups.
\newblock {\em J. Lond. Math. Soc.}, 75(1):47--66, 2007.

\bibitem{mclaughlin}
J.~McLaughlin.
\newblock {\em Random Walks and Convolution Operators on Free \mbox{Products}}.
\newblock PhD thesis, New York Univ., 1986.

\bibitem{olver}
F.~Olver.
\newblock Asymptotics and {S}pecial {F}unctions.
\newblock {\em Academic Press, San Diego, California}, 1974.

\bibitem{sawyer78}
S.~Sawyer.
\newblock Isotropic random walks in a tree.
\newblock {\em Z. Wahrscheinlichkeitstheorie}, Verw. Geb. 42:279--292, 1978.

\bibitem{voiculescu}
D.~Voiculescu.
\newblock Addition of certain non-commuting random variables.
\newblock {\em J. Funct. Anal.}, 66:323--346, 1986.

\bibitem{woess82}
W.~Woess.
\newblock A local limit theorem for random walks on certain discrete groups.
\newblock In {\em Probability measures on groups ({O}berwolfach, 1981)}, volume
  928 of {\em Lecture Notes in Math.}, pages 467--477. Springer, Berlin, 1982.

\bibitem{woess3}
W.~Woess.
\newblock Nearest neighbour random walks on free products of discrete groups.
\newblock {\em Boll. Unione Mat. Ital.}, 5-B:961--982, 1986.

\bibitem{woess}
W.~Woess.
\newblock {\em Random Walks on Infinite Graphs and Groups}.
\newblock Cambridge University Press, 2000.

\bibitem{woess03}
W.~Woess.
\newblock Generating function techniques for random walks on graphs, in
  ``{H}eat {K}ernels and {A}nalysis on {M}anifolds, {G}raphs, and {M}etric
  {S}paces''.
\newblock {\em Contemp. Math.}, 338:391--423, 2003.

\end{thebibliography}

\end{document}